\documentclass[11pt,a4paper]{article}
\usepackage[latin1]{inputenc}
\usepackage{amsfonts,amsmath,amssymb}
\usepackage{graphicx}
\usepackage{multicol}
\usepackage{epic}
 \advance\textwidth30mm
 \advance\hoffset-15mm
 \advance\textheight20mm
 \advance\voffset-15mm
 \def\<{\left<}
 \def\>{\right>}
 
 \let\fle\to

 \def\e{\varepsilon}

 \def\R{\mathbb{R}}
 \def\L{\mathbb{L}}
 \def\S{\mathbb{S}}
 \def\H{\mathbb{H}}
 \def\tr{\text{\rm tr}}
 \def\det{\text{\rm det}}
 \def\cm{{\mathcal C}^\infty(M)}
 \let\ds\displaystyle
 \def\div{\text{\rm div}}
 \newcommand{\U}[1][k+1]{\mathcal{U}_{#1}}
 \newcommand{\V}[1][k]{\mathcal{V}_{#1}}
 \def\card{\text{\rm card}}
 \newcommand{\ol}{\overline}

 \newcommand{\va}{a}
 \newcommand{\vb}{b}
 \newcommand{\vA}{A}
 \newcommand{\mb}{\mathbb{M}_c^{n+1}}

\def\dashcruz#1#2#3#4{\setlength{\unitlength}{1mm}%
 \begin{picture}(0,0)
 \dashline[50]{2}(#1,0)(#2,0)
 \dashline[50]{2}(0,#3)(0,#4)
 \end{picture}}

\def\MatCero{\text{\Large\bfseries 0}}
\def\ColSep{\kern2pt}
\def\Sep#1{\kern#1pt}
\def\VSep{5pt}

\newtheorem{teor}{Theorem}
\newtheorem{lema}[teor]{Lemma}

\newtheorem{prop}[teor]{Proposition}

\newtheorem{exa}{Example}
\newtheorem{cla}{Claim}
\newenvironment{ejem}{\begin{exa}\upshape}{\end{exa}}
\newenvironment{proof}{\par\noindent Proof.\space}{\hfill$\Box$}
\newenvironment{claim}{\begin{cla}\itshape}{\end{cla}}

\title{Hypersurfaces in non-flat Lorentzian space forms satisfying $L_k\psi=A\psi+b$\kern2mm\footnote{This work has been partially supported by MICINN Project No.
MTM2009-10418, and Fundación Séneca, Spain Project No. 04540/GERM/06. This research is a result of the activity developed within the framework of the Programme
in Support of Excellence Groups of the Región de Murcia, Spain, by Fundación Séneca, Regional Agency for Science and Technology (Regional Plan for Science and
Technology 2007-2010).}}
\author{\normalsize Pascual Lucas\footnote{Corresponding author.\newline \hspace*{17pt}E-mail addresses: plucas@um.es and hectorfabian.ramirez@um.es}\and\normalsize
H.~Fabián Ramírez-Ospina}
\date{\normalsize Departamento de Matemáticas, Universidad de Murcia\\
       Campus de Espinardo, 30100 Murcia SPAIN\\[10mm] \today}

\begin{document}

\maketitle

\begin{abstract}
We study hypersurfaces either in the De Sitter space $\S_1^{n+1}\subset\R_1^{n+2}$ or in the anti De Sitter space $\H_1^{n+1}\subset\R_2^{n+2}$ whose position vector $\psi$ satisfies the condition $L_k\psi=A\psi+b$, where $L_k$ is the linearized operator of the $(k+1)$-th mean curvature of the hypersurface, for a fixed $k=0,\dots,n-1$, $A$ is an $(n+2)\times(n+2)$ constant matrix and $b$ is a constant vector in the corresponding pseudo-Euclidean space. For every $k$, we prove that when $A$ is self-adjoint and $b=0$, the only hypersurfaces satisfying that condition are hypersurfaces with zero $(k+1)$-th mean curvature and constant $k$-th mean curvature, open pieces of standard pseudo-Riemannian products in $\S_1^{n+1}$ ($\S_1^m(r)\times\S^{n-m}(\sqrt{1-r^2})$, $\H^m(-r)\times\S^{n-m}(\sqrt{1+r^2})$, $\S_1^m(\sqrt{1-r^2})\times\S^{n-m}(r)$, $\H^m(-\sqrt{r^2-1})\times\S^{n-m}(r)$), open pieces of standard pseudo-Riemannian products in $\H_1^{n+1}$ ($\H_1^m(-r)\times\S^{n-m}(\sqrt{r^2-1})$, $\H^m(-\sqrt{1+r^2})\times\S_1^{n-m}(r)$, $\S_1^m(\sqrt{r^2-1})\times\H^{n-m}(-r)$, $\H^m(-\sqrt{1-r^2})\times\H^{n-m}(-r)$) and open pieces of a quadratic hypersurface $\{x\in\mathbb{M}_{c}^{n+1}\;|\;\<Rx,x\>=d\}$, where $R$ is a self-adjoint constant matrix whose minimal polynomial is $t^2+at+b$, $a^2-4b\leq 0$, and $\mathbb{M}_{c}^{n+1}$ stands for $\S_1^{n+1}\subset\R_1^{n+2}$ or  $\H_1^{n+1}\subset\R_2^{n+2}$. When $H_k$ is constant and $b$ is a non-zero constant vector, we show that the hypersurface is totally umbilical, and then we also obtain a classification result (see Theorem 2).
\bigskip

\noindent\textbf{Mathematics Subject Classifications (2010):} 53C50, 53B25, 53B30\bigskip

\noindent\textbf{Keywords:} linearized operator $L_k$; isoparametric hypersurface; $k$-maximal hypersurface; Takahashi theorem; higher order mean curvatures; Newton transformations.
\end{abstract}

\clearpage

\section{Introduction}
\label{s:introduction}

It is well known that the Laplacian operator of a hypersurface $M^n$ immersed into $\R^{n+1}$ is an (intrinsic) second-order linear differential operator, which arises naturally as the
linearized operator of the first variation of the mean curvature for normal variations of the hypersurface. From this point of view, the Laplacian operator $\Delta$ can be seen as
the first one of a sequence of operators $\{L_0 = \Delta$, $L_1,\dots,L_{n-1}\}$, where $L_k$ stands for the linearized operator of the first variation of the $(k + 1)$th mean curvature, arising
from normal variations of the hypersurface (see, for instance, \cite{Reilly}). These operators are given by $L_k(f)=\tr(P_k\circ\nabla^2f)$, for a smooth function $f$ on $M$, where $P_k$
denotes the $k$th Newton transformation associated to the second fundamental form of the hypersurface, and $\nabla^2f$ denotes the self-adjoint linear operator metrically equivalent to the hessian of $f$. In particular, when $k=1$ the operator $L_1$ is nothing but the operator $\Box$ introduced by Cheng and Yau in \cite{CY} for the study of hypersurfaces with constant scalar curvature. Note that, in this context, the scalar curvature of $M$ is nothing but $\e n(n-1)H_2$, where $H_2$ stands for the second mean curvature and $\e=\pm1$ depends on the causal character of the normal vector (see next section for details).

From this point of view, and inspired by Garay's extension of Takahashi theorem and its subsequent generalizations and extensions (\cite{Tak66}, \cite{CP}, \cite{Gar90}, \cite{DPV90}, \cite{HV92}, \cite{AFL92-a}, \cite{AFL92-b}, \cite{AFL95}), Alías and Gürbüz initiated in \cite{AG06} the study of hypersurfaces in Euclidean space satisfying the general condition $L_k\psi=A\psi+b$, where $A\in\R^{(n+1)×(n+1)}$ is a constant matrix and $b\in\R^{n+1}$ is a constant vector. They show that the only hypersurfaces satisfying that condition are open pieces of hypersurfaces with zero $(k+1)$-th mean curvature, or open pieces of a round sphere $\S^n(r)$, or open pieces of a generalized spherical cylinder $\S^m(r)\times\R^{n-m}$, with $k+1\leq m\leq n-1$. Following the ideas contained in \cite{AG06}, we have completely extended to the Lorentz-Minkowski space the previous classification theorem obtained by Alías and Gürbüz. In particular, the following classification result was given in \cite[Theorem 1]{LR2010a}.

\noindent\textbf{Theorem A. (\cite{LR2010a})} {\itshape Let $\psi:M\rightarrow\mathbb{L}^{n+1}$ be an orientable hypersurface immersed into the Lorentz-Minkowski space $\mathbb{L}^{n+1}$, and let
$L_k$ be the linearized operator of the $(k+1)$th mean curvature of $M$, for some fixed $k=0,1,\dots,n-1$. Then the immersion satisfies the
condition $L_k\psi=A\psi+b$, for some constant matrix $A\in\mathbb{R}^{(n+1)\times(n+1)}$ and some constant vector $b\in\mathbb{L}^{n+1}$, if and only if it is one of the following hypersurfaces in $\mathbb{L}^{n+1}$:\vspace*{-.75\baselineskip}
\begin{enumerate}\itemsep-\parskip
\item a hypersurface with zero $(k+1)$th mean curvature;
\item an open piece of the totally umbilical hypersurface $\S^n_1(r)$ or $\H^n(-r)$;
\item an open piece of a generalized cylinder $\S^m_1(r)\times\R^{n-m}$, $\H^m(-r)\times\R^{n-m}$, with $k+1\leq m\leq n-1$, or $\L^m\times\S^{n-m}(r)$, with $k+1\leq n-m\leq n-1$.
\end{enumerate}}

In \cite{AK}, and as a natural continuation of the study started in \cite{AG06}, Alías and Kashani consider the study of hypersurfaces $M^n$ immersed either into the sphere $\S^{n+1}\subset\R^{n+2}$ or into the hyperbolic space $\H^{n+1}\subset\R^{n+2}_1$ whose position vector $x$ satisfies the condition $L_kx=Ax+b$, for some constant matrix $A\in\mathbb{R}^{(n+2)\times(n+2)}$ and some constant vector $b\in\mathbb{R}^{n+2}_q$, $q=0,1$. They show the following two results:

\noindent\textbf{Theorem B. (\cite{AK})} {\itshape The immersion $x$ satisfies the condition $L_kx=Ax$, for some self-adjoint constant matrix $A\in\mathbb{R}^{(n+2)\times(n+2)}$, if and only if it is one of the following hypersurfaces: (1) a hypersurface having zero (k+1)-th mean curvature and constant k-th mean curvature; (2) an open piece of a standard Riemannian product $\S^m(\sqrt{1-r^2})\times\S^{n-m}(r)\subset\S^{n+1}$, $0<r<1$; (3) an open piece of a standard Riemannian product $\H^m(-\sqrt{1+r^2})\times\S^{n-m}(r)\subset\H^{n+1}$, $r>0$.}

\noindent\textbf{Theorem C. (\cite{AK})} {\itshape The immersion $x$ satisfies the condition $L_kx=Ax+b$, for some self-adjoint constant matrix $A\in\mathbb{R}^{(n+2)\times(n+2)}$ and some non-zero constant vector $b\in\R^{n+2}$, if and only if it is one of the following hypersurfaces: (1) an open piece of a totally umbilical round sphere $\S^n(r)\subset\S^{n+1}$; (2) an open piece of a totally umbilical hyperbolic space $\H^n(-r)\subset\H^{n+1}$, $r>1$; (3) an open piece of a totally umbilical round sphere $\S^n(r)\subset\H^{n+1}$, $r>0$; (4) an open piece of a totally umbilical Euclidean space $\R^n\subset\H^{n+1}$.}

The hypersurfaces studied in Theorems B and C are Riemannian, and thus their shape operators are always diagonalizable. However, when the ambient space is a Lorentzian space form $\S^{n+1}_1$ or $\H^{n+1}_1$, the shape operator of the hypersurface needs not be diagonalizable, condition which plays a chief role in the Riemannian case. In this paper we extend, to the indefinite case, the results obtained in \cite{AK} for hypersurfaces immersed either into the sphere or into the hyperbolic space. For the sake of simplifying the notation and unifying the statements of our main results, let us denote by $\mathbb{M}^{n+1}_c$ either the De Sitter space $\S^{n+1}_1\subset\R^{n+2}_1$ if $c=1$, or the anti De Sitter space $\H^{n+1}_1\subset\R^{n+2}_2$ if $c=-1$. In this paper, we are able to give the following classification result.

\begin{teor}\label{T1}
Let $\psi:M\rightarrow\mathbb{M}_{c}^{n+1}\subset\mathbb{R}_q^{n+2}$ be an orientable hypersurface immersed into the space form $\mathbb{M}_c^{n+1}$, and let $L_k$ be the linearized operator of the $(k+1)$-th mean curvature of $M$, for some fixed $k=0,1,\ldots,n-1$. Then the immersion satisfies the condition $L_k\psi=A\psi$, for some self-adjoint constant matrix $A\in\mathbb{R}^{(n+2)\times(n+2)}$, if and only if it is one of the following hypersurfaces:\vspace*{-.8\baselineskip}
\begin{enumerate}\itemsep-2pt
\item[\upshape (1)] a hypersurface having zero $(k+1)$-th mean curvature and constant $k$-th mean curvature;
\item[\upshape (2)] an open piece of a standard pseudo-Riemannian product in $\S_1^{n+1}$: $\S_1^m(r)\times\S^{n-m}(\sqrt{1-r^2})$, $\H^m(-r)\times\S^{n-m}(\sqrt{1+r^2})$, $\H^m(-\sqrt{r^2-1})\times\S^{n-m}(r)$.
\item[\upshape (3)] an open piece of a standard pseudo-Riemannian product in $\H_1^{n+1}$: $\H_1^m(-r)\times\S^{n-m}(\sqrt{r^2-1})$, $\H^m(-\sqrt{1+r^2})\times\S_1^{n-m}(r)$, $\S_1^m(\sqrt{r^2-1})\times\H^{n-m}(-r)$, $\H^m(-\sqrt{1-r^2})\times\H^{n-m}(-r)$.
\item[\upshape (4)] an open piece of a quadratic hypersurface $\{x\in\mathbb{M}_{c}^{n+1}\subset\mathbb{R}_q^{n+2}\;|\;\<Rx,x\>=d\}$, where $R$ is a self-adjoint constant matrix whose minimal polynomial is $t^2+at+b$, $a^2-4b\leq 0$.
\end{enumerate}
\end{teor}

Finally, in the case where $A$ is self-adjoint and $b$ is a non-zero constant vector, we are able to prove the following classification result.

\begin{teor}\label{T2}
Let $\psi:M\rightarrow\mathbb{M}_{c}^{n+1}\subset\mathbb{R}_q^{n+2}$ be an orientable hypersurface immersed into the space form $\mathbb{M}_c^{n+1}$, and let $L_k$ be the linearized operator of the $(k+1)$-th mean curvature of $M$, for some fixed $k=0,1,\ldots,n-1$. Assume that $H_k$ is constant. Then the immersion satisfies the condition $L_k\psi=\vA\psi+\vb$, for some self-adjoint constant matrix $\vA\in\mathbb{R}^{(n+2)\times(n+2)}$ and some non-zero constant vector $\vb\in\mathbb{R}^{n+2}_q$, if and only if:\vspace*{-.8\baselineskip}
\begin{enumerate}\itemsep-5pt
\item[\upshape (i)] $c=1$ and it is an open piece of a totally umbilical hypersurface in $\S^{n+1}_1\subset\mathbb{R}^{n+2}_1$: $\S^n(r)$, $r>1$; $\H^n(-r)$, $r>0$; $\S^n_1(r)$, $0<r<1$; $\R^n$.
\item[\upshape (ii)] $c=-1$ and it is an open piece of a totally umbilical hypersurface in $\H^{n+1}_1\subset\mathbb{R}^{n+2}_2$: $\H^n_1(-r)$, $r>1$; $\H^n(-r)$, $0<r<1$; $\S^n_1(r)$, $r>0$; $\R^n_1$.
\end{enumerate}
\end{teor}

\section{Preliminaries}
\label{s:preliminaries}

In this section we recall some formulas and notions about hypersurfaces in Lorentzian space forms that will be used later on. Let $\mathbb{R}^{n+2}_{q}$ be the $(n+2)$-dimensional pseudo-Euclidean space of index $q\geq 1$, whose metric tensor $\<,\>$ is given by
\[
\<,\>=-\sum_{i=1}^q dx_i^2+\sum_{j=q+1}^{n+2} dx_j^2,
\]
where $x=(x_1,\ldots,x_{n+2})$ denotes the usual rectangular coordinates in $\mathbb{R}^{n+2}$. The pseudo-Euclidean De Sitter space of index $q$ and radius $r$ is defined by
\[
\mathbb{S}^{n+1}_{q}(r)=\{x\in \mathbb{R}^{n+2}_{q} \;|\; \<x,x\>=r^2\},
\]
and the pseudo-Euclidean anti-De Sitter space of index $q$ and radius $-r$ is defined by
\[
\mathbb{H}^{n+1}_{q}(-r)=\{x\in \mathbb{R}^{n+2}_{q+1} \;|\; \<x,x\>=-r^2\}.
\]
Throughout this paper, we will consider both the case of hypersurfaces immersed into Lorentzian De Sitter space $\mathbb{S}^{n+1}_{1}\equiv\mathbb{S}^{n+1}_{1}(1)$, and the case of hypersurfaces immersed into Lorentzian anti De Sitter space $\mathbb{H}^{n+1}_{1}\equiv\mathbb{H}^{n+1}_{1}(-1)$. In order to simplify our notation and computations, we will denote by $\mathbb{M}^{n+1}_c$ the De Sitter space $\mathbb{S}^{n+1}_{1}$ or the anti De Sitter space $\mathbb{H}^{n+1}_{1}$ according to $c=1$ or $c=-1$, respectively. We will use $\mathbb{R}^{n+2}_{q}$  to denote the corresponding pseudo-Euclidean space where $\mathbb{M}^{n+1}_{c}$ lives, so that $q=1$ if $c=1$ and $q=2$ if $c=-1$. Then its metric is given by
\[
\<,\>=-dx_1^2+cdx_2^2+dx_3^2+\cdots+dx_{n+2}^2,
\]
and we can write
\[
\mathbb{M}^{n+1}_c=\{x\in \mathbb{R}^{n+2}_q \;|\; -x_1^2+cx_2^2+x_3^2+\cdots+x^2_{n+1}=c\}.
\]
It is well known that $\mathbb{S}^{n+1}_1\subset\mathbb{R}^{n+2}_1$ and $\mathbb{H}^{n+1}_1\subset\mathbb{R}^{n+2}_2$ are Lorentzian totally umbilical hypersurfaces with constant sectional curvature $+1$ and $-1$, respectively.

Let $\psi:M\longrightarrow \mathbb{M}^{n+1}_c\subset\mathbb{R}^{n+2}_{q}$ be a connected orientable hypersurface with Gauss map $N$, $\<N,N\>=\e=\pm1$. Let $\nabla^0$, $\ol\nabla$ and $\nabla$ denote the Levi-Civita connections on $\mathbb{R}^{n+2}_{q}$, $\mathbb{M}^{n+1}_c$ and $M$, respectively. Then the Gauss and Weingarten formulas are given by
\begin{equation}\label{F.G}
\nabla^0_XY=\nabla_XY+\varepsilon\< SX,Y\> N-c\< X,Y\> \psi,
\end{equation}
and
\[
SX=-\ol{\nabla}_XN=-\nabla^0_XN,
\]
for all tangent vector fields $X,Y\in\mathfrak{X}(M)$, where $S:\mathfrak{X}(M)\longrightarrow \mathfrak{X}(M)$ stands for the shape operator (or Weingarten endomorphism) of $M$, with respect to the chosen orientation $N$.

Let ${\mathcal B}=\{E_1,E_2,\ldots,E_{n+1}\}$ be a (local) frame in $\mathbb{M}^{n+1}_c$. Without loss of generality, we will say that ${\mathcal B}$ is an
orthornormal frame when
\begin{equation*}
\begin{array}{l}
\< E_1,E_1\>=-1 \text{ and }
\< E_1,E_j\>=0,\quad j=2,\ldots,{n+1},\\
\< E_i,E_j\>=\delta_{ij},\quad 2\leq i,j\leq{n+1};
\end{array}
\end{equation*}
and we will say that ${\mathcal B}$ is a pseudo-orthornormal frame, when the following conditions are satisfied:
\begin{equation*}
\begin{array}{l}
\< E_1,E_2\>=-1 \text{ and }
\< E_1,E_1\>=\< E_2,E_2\>=0,\\
\< E_i,E_j\>=0,\quad i=1,2,\quad j=3,\ldots,{n+1},\\
\< E_i,E_j\>=\delta_{ij},\quad 3\leq i,j\leq{n+1}.
\end{array}
\end{equation*}

It is well-known (see, for instance, \cite[pp. 261--262]{ONeill}) that the shape operator $S$ of the hypersurface $M$ can be expressed, in an appropriate frame, in one of the following types:
\begin{align}\label{Formas}
 & \text{I. } S\approx
 \left[\begin{array}{@{\ColSep}c@{\ColSep}c@{\ColSep}c@{\ColSep}c@{\ColSep}}
    \kappa_1 &         &        & \MatCero  \\
             &\kappa_2 &        &           \\
             &         & \ddots &           \\
    \MatCero &         &        & \kappa_{n}
 \end{array}\right];
 \kern5mm
 \text{II. } S\approx
 \left[\begin{array}{@{\ColSep}c@{\Sep4}r@{\Sep4}c@{\Sep4}c@{\ColSep}c@{\ColSep}c@{\ColSep}}
   \kappa &  -b   &                             & & & \MatCero \\
   b      &\kappa &                             & & &          \\[-\VSep]
          &       & \dashcruz{-10}{16}{-15}{11} & & & \\[-\VSep]
          &       &                             & \kappa_3 & & \\
          &       &                             & & \ddots &   \\
   \MatCero&    &                             & & &\kappa_{n}
   \end{array}\right],\quad b\neq 0;
   \nonumber\\[5mm]
 & \text{III. } S\approx
 \left[\begin{array}{@{\ColSep}c@{\Sep6}c@{\Sep4}c@{\Sep4}c@{\ColSep}c@{\ColSep}c@{\ColSep}}
   \kappa &  0    &                             & & & \MatCero \\
   1      &\kappa &                             & & &          \\[-\VSep]
          &       & \dashcruz{-9}{16}{-15}{11} & & & \\[-\VSep]
          &       &                             & \kappa_3 & & \\
          &       &                             & & \ddots &   \\
   \MatCero&    &                             & & &\kappa_{n}
   \end{array}\right];
 \kern5mm
 \text{IV. } S\approx
 \left[\begin{array}{@{\ColSep}r@{\Sep6}c@{\Sep6}c@{\Sep4}c@{\Sep4}c@{\ColSep}c@{\ColSep}c@{\ColSep}}
  \kappa & 0      & 0      &                             &        &  & \MatCero \\
    0    & \kappa & 1      &                             &        & \\
    -1   & 0      & \kappa &                             &        & \\[-\VSep]
         &        &        & \dashcruz{-16}{16}{-16}{16} &        & & \\[-\VSep]
         &        &        &                             &\kappa_4 & & \\
         &        &        &                             & & \ddots & \\
    \MatCero &    &        &                             & & & \kappa_{n}
   \end{array}\right].
\end{align}
In cases I and II, $S$ is represented with respect to an orthonormal frame, whereas in cases III and IV, the frame is pseudo-orthonormal.

The characteristic polynomial $Q_S(t)$ of the shape operator $S$ is given by
\begin{align*}
	Q_S(t)&={\rm det}(tI-S)=\sum_{k=0}^na_kt^{n-k},\quad\text{ with } a_0=1.
\end{align*}
Making use of the Leverrier--Faddeev method (see \cite{Lev,Fad}), the coefficients of $Q_S(t)$ can be computed, in terms of the traces of $S^j$, as follows:
\begin{equation}\label{C.Pol}
    a_k=-\frac{1}{k}\sum_{j=1}^ka_{k-j}{\rm{tr}}(S^j),\quad k=1,\ldots,n,\quad\text{ with } a_0=1.	
\end{equation}
Bearing in mind the type of shape operator $S$, we can see that the coefficients of $Q_S(t)$ for $S$ of types I, III and IV, are given by
\begin{equation}\label{CPC1}
\left\{\begin{array}{l}
\displaystyle a_1=-\sum_{i=1}^n\kappa_i,\\
\displaystyle a_k=(-1)^k \kern-6pt \sum^n_{i_1<\cdots<i_k}\kern-8pt \kappa_{i_1}\cdots\kappa_{i_k},\quad  k=2,\ldots,n,
\end{array}\right.
\end{equation}
whereas if $S$ is of type II then they are given by
\begin{equation}\label{CPC2}
\left\{\begin{array}{ll}
\ds a_1=-\sum_{i=1}^n\kappa_i,\\
\ds a_k=(-1)^k\Bigg[\sum^n_{i_1<\cdots<i_k}\kern-8pt\kappa_{i_1}\cdots\kappa_{i_k}\  +\ b^2\kern-12pt \sum^n_{\textrm{\tiny{$\begin{array}{cc}i_1\!\!<\!\!\cdots\!<\!i_{k-2}\\ i_j\neq 1,2\end{array}$}}}
\kern-12pt\kappa_{i_1}\cdots\kappa_{i_{k-2}}\Bigg],\quad k=2,\ldots,n.
\end{array}\right.
\end{equation}
If $S$ is of type II or III, then we consider that $\kappa_1=\kappa_2=\kappa$, and if $S$ is of type IV we consider that $\kappa_1=\kappa_2=\kappa_3=\kappa$. From now on, we will write
\begin{equation*}
\mu_{_{k}}=\kern-5pt\sum^n_{i_1<\cdots<i_k}\kern-7pt\kappa_{i_1}\cdots\kappa_{i_k} \qquad\textrm{and}\qquad\mu^{J}_{_{k}}=\kern-9pt \sum^n_{\tiny{\begin{array}{cc}i_1\!\!<\!\!\cdots\!<\!i_{k}\\ i_j\notin J \end{array}}}
\kern-13pt\kappa_{i_1}\cdots\kappa_{i_k},
\end{equation*}
where $k\in\{1,\ldots,n\}$ and $J\subset\{1,\ldots,n\}$. Observe that
\begin{equation}\label{F}
\mu_{_{k}}^\emptyset=\mu_{_{k}}\quad\text{and}\quad\mu_{_{k}}=\kappa_m\mu_{_{k-1}}^m+\mu_{_{k}}^m,
\end{equation}
where $\mu_k^m$ stands for $\mu_k^{\{m\}}$.

Then the coefficients $a_k$ of characteristic polynomial $Q_S(t)$, given in equations (\ref{CPC1}) and (\ref{CPC2}), can be easily written as follows
\begin{align}\label{F2}
 a_k&=(-1)^k\mu_{_{k}}, &\textrm{in cases I, III, IV;}\\
 a_k&=(-1)^k(\mu_{_{k}}+b^2\mu_{_{k-2}}^{1,2}), &\textrm{in case II.}\label{F3}
\end{align}
We use here that $\mu_{_{0}}=1$ and $\mu_{_{k}}=0$ if $k<0$.

The \emph{$k$-th mean curvature} or \emph{mean curvature of order $k$} of $M$ is defined by
\begin{equation}\label{C.H_k}
\binom{n}{k}H_k=(-\varepsilon)^ka_k,
\end{equation}
where $\ds\binom{n}{k}=\frac{n!}{k!(n-k)!}$. In particular, when $k=1$,
\[
nH_1=-\varepsilon a_1=\varepsilon{\rm tr}(S),
\]
and so $H_1$ is nothing but the usual mean curvature $H$ of $M$, which is one of the most important extrinsic curvatures of the hypersurface. The hypersurface $M$ is said to be $k$-maximal in $\mathbb{M}^{n+1}_c$ if $H_{k+1}\equiv 0$. On the other hand, $H_2$ defines a geometric quantity which is related to
the (intrinsic) scalar curvature of $M$. Indeed, it follows from the Gauss equation of $M$
that its Ricci curvature is given by
\begin{equation}\label{C.Ricci}
\mathrm{Ric}(X,Y)=(n-1)c\<X,Y\>+nH_1\<SX,Y\>-\varepsilon\<SX,SY\>,\quad X,Y\in\mathfrak{X}(M),
\end{equation}
and then, from (\ref{C.Pol}), the scalar curvature Scal=tr(Ric) of $M$ is
\begin{equation}\label{C.Escalar}
{\rm Scal}=n(n-1)c+\varepsilon\Big(-a_1 \mathrm{tr}(S)-\mathrm{tr}(S^2)\Big)=n(n-1)(c+\varepsilon H_2).
\end{equation}

\section{The Newton transformations}
\label{s:Newton}

The $k$-th Newton transformation of $M$ is the operator $P_k:\mathfrak{X}(M)\longrightarrow \mathfrak{X}(M)$ defined by
\begin{align*}
P_k=\sum_{j=0}^ka_{k-j}S^j.
\end{align*}
Equivalently, $P_k$ can be defined inductively by
\begin{align}\label{indutiva}
P_0=I\quad\text{and}\quad P_k=a_kI+S\circ P_{k-1}.
\end{align}
Note that by Cayley-Hamilton theorem we have $P_n=0$. The Newton transformations were introduced by Reilly \cite{Reilly} in the Riemannian context; its definition was $\ol{P}_k=(-1)^k P_k$. We have the following properties of $P_k$ (the proof is algebraic and straightforward).

\begin{lema}\label{L1}\label{Pro}
Let $\psi:M^n\rightarrow\mathbb{M}^{n+1}_c$ be a hypersurface in the Lorentzian space form $\mathbb{M}^{n+1}_c$. The Newton transformations $P_k$ satisfy:\vspace*{-.75\baselineskip}
\begin{list}{}{\leftmargin5pt\itemsep0pt}
\item[(a)] $P_k$ is self-adjoint and commutes with $S$.
\item[(b)] $\tr(P_k)=(n-k)a_k=c_kH_k$.
\item[(c)] $\tr(S\circ P_k)=-(k+1)a_{k+1}=\e c_kH_{k+1}$, $1\leq k\leq n-1$.
\item[(d)] $\tr(S^2\circ P_{k})=a_1a_{k+1}-(k+2)a_{k+2}=C_k\big(nH_1H_{k+1}-(n-k-1)H_{k+2}\big)$, $1\leq k\leq n-2$.
\end{list}\vspace*{-.75\baselineskip}
Here, the constants $c_k$ and $C_k$ are given by
\[
c_k=(-\varepsilon)^k(n -k)\binom{n}{k}=(-\varepsilon)^k(k+1)\binom{n}{k+1}\qquad\textrm{and}\qquad C_k=\dfrac{c_k}{k+1}.
\]
\end{lema}

Next we are going to describe the covariant derivative of the shape operator $S$ and the $k$th Newton transformation $P_k$. To do that, we will work with a (local) tangent frame of vector fields $\{E_1,E_2\ldots,E_n\}$ in which $S$ adopts its canonical form, and we need to distinguish four cases, according to the canonical form of the shape operator, see equation (\ref{Formas}).

Let $(w_i^j)$ be the connection 1-forms, defined by $w_i^j(X)=\<\nabla_XE_i,E_j\>$, so that $w_i^j=-w_j^i$. The following four propositions are technical results that we will use later on. Their proofs are straightforward.

\begin{prop}[\emph{S} is of type I]\label{caso1}\ \\
Suppose that the shape operator $S$ is of type I, and let $\{E_1,E_2\ldots,E_n\}$ be an orthonormal frame such that $SE_i=\kappa_iE_i$,  $i=1,\ldots,n$. Then we have:
\begin{align*}
(\nabla_XS)E_i &=X(\kappa_i)E_i+\sum_{i\neq j}\varepsilon_{\!j}\,(\kappa_i-\kappa_j)\omega_i^j(X)E_j,\\
P_{k}E_i &=(-1)^k\mu^i_{_{k}}E_i,
\end{align*}
for every $i=1,\dots,n$, where $\e_i=\<E_i,E_i\>$.
\end{prop}

\begin{prop}[\emph{S} is of type II]\label{caso2}\ \\
Suppose that the shape operator $S$ is of type II, and let $\{E_1,E_2\ldots,E_n\}$ be an orthonormal frame such that $SE_1=\kappa E_1+bE_2$, $SE_2=-bE_1+\kappa E_2$, and $SE_i=\kappa_iE_i$, $i\geq3$. Then the covariant derivative $\nabla S$ is given by
\begin{align*}
(\nabla_XS)E_1&=\left(X(\kappa)+2b\,\omega_1^2(X)\right)E_1 +X(b)E_2+ \sum^n_{j=3}\left((\kappa-\kappa_j)\,\omega_1^j(X)+b\,\omega_2^j(X)\right)E_j,\\
(\nabla_XS)E_2&=-X(b)E_1+\left(X(\kappa)+2b\,\omega_2^1(X)\right)E_2 +\sum^n_{j=3}\left((\kappa-\kappa_j)\,\omega_2^j(X)+b\,\omega_1^j(X)\right)E_j,\\
(\nabla_XS)E_i&=\left(\kappa\,\omega_i^1(X)+b\,\omega_i^2(X)- \kappa_i\omega_i^1(X)\right)E_1
+\left(b\,\omega_i^1(X)-\kappa\,\omega_i^2(X)+\kappa_i\omega_i^2(X)\right)E_2\\
&\quad+X(\kappa_i)E_i+\sum^n_{j\neq1,2,i}(\kappa_i-\kappa_j)\,\omega_i^j(X)E_j,\quad i\geq3.
\end{align*}
The Newton transformation $P_k$ satisfies
\begin{align*}
P_kE_1&=(-1)^k(\mu_{_k}^1E_1+b\mu_{_{k-1}}^{1,2}E_2),\\
P_kE_2&=(-1)^k(-b\mu_{_{k-1}}^{1,2}E_1+\mu_{_k}^1E_2),\\
P_kE_i&=(-1)^k(\mu^i_{_{k}}+b^2\mu^{1,2,i}_{_{k-2}})E_i,\quad i\geq3.
\end{align*}
\end{prop}

\begin{prop}[\emph{S} is of type III]\label{caso3}\ \\
Assume that the shape operator $S$ is of type III, and let $\{E_1,E_2,\ldots,E_n\}$  be a pseudo-orthonormal frame such that $SE_1=\kappa E_1+E_2$, $SE_2=\kappa E_2$, and $SE_i=\kappa_iE_i$, $i\geq3$. Then the covariant derivative $\nabla S$ satisfies
\begin{align*}
(\nabla_XS)E_1&=X(\kappa)E_1+2\,\omega_1^2(X)E_2+\sum^n_{j=3}\left((\kappa-\kappa_j)\,\omega_1^j(X)+\omega_2^j(X)\right)E_j,\\
(\nabla_XS)E_2&=X(\kappa)E_2+\sum^n_{j=3}(\kappa-\kappa_j)\,\omega_2^j(X)E_j,\\
(\nabla_XS)E_i&=\big(\kappa\omega_i^2(X)-\kappa_i\omega_i^2(X)\big)E_1+\big(\omega_i^2(X)+\kappa\omega_i^1(X)-\kappa_i\omega_i^1(X)\big)E_2\\
&\quad +X(\kappa_i)E_i+\sum^n_{j\neq1,2,i}(\kappa_i-\kappa_j)\,\omega_i^j(X)E_j,\quad i\geq3.
\end{align*}
The Newton transformation $P_k$ is given by
\begin{align*}
  P_kE_1&=(-1)^k\big(\mu_{_k}^1E_1-\mu_{_{k-1}}^{1,2}E_2\big),\\
  P_kE_2&=(-1)^k\mu_{_k}^1E_2,\\
  P_kE_i&=(-1)^k\mu^i_{_{k}}E_i,\quad i\geq3.
   \end{align*}
\end{prop}

\begin{prop}[\emph{S} is of type IV]\label{caso4}\ \\
Suppose that the shape operator $S$ is of type IV, and let $\{E_1,E_2,\ldots,E_n\}$ be a pseudo-orthonormal tangent frame such that $SE_1=\kappa E_1-E_3$, $SE_2=\kappa E_2$, $SE_3=E_2+\kappa E_3$, and $SE_i=\kappa_iE_i$, $i\geq4$. Then the covariant derivative $\nabla S$ satisfies
\begin{align*}
(\nabla_XS)E_1&=\left(X(\kappa)+\omega_3^2(X)\right)E_1+2\,\omega_3^1(X)E_2- \omega_1^2(X)E_3+\sum^n_{j=4}\left((\kappa-\kappa_j)\,\omega_1^j(X)- \omega_3^j(X)\right)E_j,\\
(\nabla_XS)E_2&=\left(X(\kappa)-\omega_2^3(X)\right)E_2+\sum^n_{j=4}(\kappa-\kappa_j)\,\omega_i^j(X)E_j,\\ (\nabla_XS)E_3&=-\omega_1^2(X)E_2+\left[X(\kappa)+2\,\omega_2^3(X)\right]E_3+ \sum^n_{j=4}\left((\kappa-\kappa_j)\,\omega_3^j(X)+\omega_2^j(X)\right)E_j,\\
(\nabla _XS)E_i&=\left(\kappa\omega_i^2(X)-\kappa_i\omega_i^2(X)\right)E_1+\left(\kappa\omega_i^1(X)- \omega_i^3(\!X\!)-\kappa_i\omega_i^1(\!X\!)\right)E_2\\
&+\left(\kappa_i\omega_i^3(X)-\kappa\omega_i^3(X)-\omega_i^2(X)\right)E_3 +X(\kappa_i)E_i+\sum^n_{j\neq1,2,3,i}(\kappa_i-\kappa_j)\,\omega_i^j(X)E_j,\quad i\geq4.
\end{align*}
The Newton transformation $P_k$ is given by
\begin{align*}
   P_kE_1&=(-1)^k(\mu^1_{_k}E_1-\mu^{1,2,3}_{_{k-2}}E_2 +\mu^{1,2}_{_{k-1}}E_3),\\
   P_kE_2&=(-1)^k\mu_{_k}^1E_2,\\
   P_kE_3&=(-1)^k(-\mu^{1,2}_{_{k-1}}E_2+\mu^1_{_k}E_3),\\
   P_kE_i&=(-1)^k\mu^i_{_{k}}E_i,\quad i\geq4.
   \end{align*}
\end{prop}

In the following lemma we present two new properties of the Newton transformations. For any differentiable function $f\in\cm$, the gradient of $f$ is the vector field $\nabla f$ metrically equivalent to $df$, which is characterized by $\<\nabla f,X\>=X(f)$, for every differentiable vector field $X\in\mathfrak{X}(M)$. The divergence of a vector field $X$ is the differentiable function defined as the trace of operator $\nabla X$, where $\nabla X(Y):=\nabla_YX$, that is,
\[
\div(X)=\tr(\nabla X)=\sum_{i,j}g^{ij}\<\nabla_{E_i}X,E_j\>,
\]
$\{E_i\}$ being any local frame of tangent vectors fields, where $(g^{ij})$ represents the inverse of the metric $(g_{ij})=(\<E_i,E_j\>)$. Analogously, the divergence of a operator $T:\mathfrak{X}(M)\longrightarrow\mathfrak{X}(M)$ is the vector field $\div(T)\in\mathfrak{X}(M)$ defined as the trace of $\nabla T$, that is,
\[
\div(T)=\tr(\nabla T)=\sum_{i,j}g^{ij}(\nabla_{E_i}T)E_j,
\]
where $\nabla T(E_i,E_j)=(\nabla_{E_i}T)E_j$.

\begin{lema}\label{L2}
The Newton transformation $P_k$, for $k=0,\ldots,n-1$, satisfies:\\
(a) $\tr(\nabla_XS\circ P_k)=-X(a_{k+1})=-\<\nabla a_{k+1},X\>=\e C_k\<\nabla H_{k+1},X\>$.\\
(b) $\div(P_k)=0$.
\end{lema}
The proof can be found in \cite{LR2010a}.

Bearing in mind this lemma we obtain
\[
\div(P_{k}(\nabla f))=\tr\big(P_{k}\circ\nabla^2 f \big),
\]
where $\nabla^2f:\mathfrak{X}(M)\longrightarrow\mathfrak{X}(M)$ denotes the self-adjoint linear operator metrically equivalent to the Hessian of $f$, given by
\[
\<\nabla^2f(X),Y\>=\<\nabla_X(\nabla f),Y\>,\qquad X,Y\in\mathfrak{X}(M).
\]
Associated to each Newton transformation $P_k$, we can define the second-order linear differential operator $L_k:\cm\longrightarrow\cm$ given by
\begin{align}\label{E9}
L_k(f)=\tr\big(P_{k}\circ\nabla^2 f \big).	
\end{align}
When $k=0$, $L_0=\Delta$ is nothing but the Laplacian operator; when $k=1$, $L_1$ is the operator $\Box$ introduced by Chen and Yau, \cite{CY}.

An interesting property of $L_k$ is the following. For every couple of differentiable functions $f,g\in C^\infty(M)$ we have
\begin{align}\nonumber
L_k(fg)&=\div\big(P_k\circ\nabla(fg)\big)=\textrm{div}\big(P_k\circ( g\nabla f + f\nabla g)\big)\\
	   &=gL_k(f)+ fL_k(g) +2\<P_k(\nabla f),\nabla g\>.\label{Lkfg}
\end{align}

\section{Examples}
\label{s:examples}

The goal of this section is to show some examples of hypersurfaces in the Lorentzian space form $\mathbb{M}_c^{n+1}$ satisfying the condition $L_k\psi=A\psi+b$, where $A$ is a constant matrix and $b$ is a constant vector. Before that, we are going to compute $L_k$ acting on the coordinate components of the immersion $\psi$, that is, a function given by $\<a,\psi\>$, where $\va\in \mathbb{R}^{n+2}_q$ is an arbitrary fixed vector.

A direct computation shows that
\begin{equation}\label{E10-}
\nabla\<\va,\psi\>=\va^\top=\va-\varepsilon\<\va,N\>N-c\<\va,\psi\>\psi,
\end{equation}
where $\va^\top\in \mathfrak{X}(M)$ denotes the tangential component of $\va$. Taking covariant derivative in~(\ref{E10-}), and using that $\nabla^0_Xa=0$, jointly with the Gauss and Weingarten formulae, we obtain
\begin{equation}\label{E11}
\nabla_X\nabla\< \va,\psi\>=\nabla_X\va^\top=\varepsilon\< \va,N\> SX-c\< \va,\psi\> X,
\end{equation}
for every vector field $X\in \mathfrak{X}(M)$. Finally, by using (\ref{E9}) and Lemma~\ref{L1}, we find that
\begin{align}\nonumber
L_k\<\va,\psi\>&=\e\<\va,N\>\tr(P_{k}\circ S)-c\<\va,\psi\>\tr(P_{k}\circ I)\\
               &=c_kH_{k+1}\<\va,N\>-cc_kH_k\<\va,\psi\>.\label{E12}
\end{align}
Then we can compute $L_k\psi$ as follows,
\begin{align}\label{E13}
L_k\psi &=\Big(L_k(\delta_1\<\psi,e_1\>),\ldots,L_k(\delta_{n+2}\<\psi,e_{n+2}\>)\Big)\nonumber\\
		 &=c_kH_{k+1}\Big(\delta_1\<e_1,N\>,\ldots,\delta_{n+2}\<e_{n+2},N\>\Big)-cc_kH_k\Big(\delta_1\<e_1,\psi\>,\ldots,\delta_{n+2}\<e_{n+2},\psi\>\Big)\nonumber\\
		&=c_kH_{k+1}N-cc_kH_k\psi,
\end{align}
where $\{e_1,\dots,e_{n+2}\}$ stands for the standard orthonormal basis in $\mathbb{R}^{n+2}_q$ and $\delta_i=\<e_i,e_i\>$.

\begin{ejem}\label{ej1}
An easy consequence of (\ref{E13}) is that every hypersurface with $H_{k+1}\equiv0$ and constant $k$-th mean curvature $H_{k}$ trivially satisfies $L_k\psi=\vA\psi+\vb$, with $\vA=-cc_kH_kI_{n+2}\in\mathbb{R}^{(n+2)\times(n+2)}$ and $\vb=0$.
\end{ejem}

\begin{ejem}(Totally umbilical hypersurfaces in $\mathbb{M}^{n+1}_c$)\quad\label{ej2}
As is well known, totally umbilical hypersurfaces in $\mathbb{M}^{n+1}_c$ are obtained as the intersection of $\mathbb{M}^{n+1}_c$ with a hyperplane of $\mathbb{R}^{n+2}_q$, and the causal character of the hyperplane determines the type of the hypersurface. More precisely, let $\va\in \mathbb{R}^{n+2}_q$ be a non-zero constant vector with $\<\va,\va\>\in\{1,0,-1\}$, and take the differentiable function $f_{\va}:\mathbb{M}^{n+1}_c\rightarrow\mathbb{R}$ defined by $f_{\va}(x)=\<a,x\>$. It is not difficult to see that for every $\tau\in \mathbb{R}$ with $\<\va,\va\>-c\tau^2\neq 0$, the set
\[
M_{\tau}=f^{-1}_\va(\tau)=\{x\in\mathbb{M}^{n+1}_c\;|\;\<\va,x\>=\tau\}
\]
is a totally umbilical hypersurface in $\mathbb{M}^{n+1}_c$, with Gauss map
\[
N(x)=\frac{1}{\sqrt{|\<\va,\va\>-c\tau^2|}}\ (a-c\tau x),
\]
and shape operator
\begin{equation}\label{B1+}
SX=-\nabla^0_XN=\frac{c\tau}{\sqrt{|\<\va,\va\>-c\tau^2|}} X.
\end{equation}
Now, by using (\ref{C.H_k}) and (\ref{F2}), we obtain that the $k$-th mean curvature is given by
\begin{equation}\label{B}
H_{k}=\frac{(\e c \tau)^k}{|\<\va,\va\>-c\tau^2|^{k/2}},\quad k=0,\ldots,n,
\end{equation}
where $\e=\<N,N\>=\pm1$. Therefore, by equation (\ref{E13}), we see that $M_\tau$ satisfies the condition $L_k\psi=\vA\psi+\vb$, for every $k=0,\ldots,n-1$, with
\[
\vA=-\frac{c_k(\e c\tau)^{k}\big(\e\tau^2+c|\<\va,\va\>-c\tau^2|\big)}{|\<\va,\va\>-c\tau^2|^{(k+2)/2}} I_{n+2}\quad\textrm{and}\quad \vb=\frac{c_k(\e c\tau)^{k+1}}{|\<\va,\va\>-c\tau^2|^{(k+2)/2}} \va.
\]
In particular, $b=0$ only when $\tau=0$, and then $M_0$ is a totally geodesic hypersurface in $\mathbb{M}^{n+1}_c$.

It is easy to see, from~(\ref{B1+}), that $M_\tau$ has constant curvature
\[
K=c+\dfrac{\tau^2}{\< \va,\va\>-c\tau^2},
\]
and it is a Riemannian or Lorentzian hypersurface according to $\< \va,\va\>-c\tau^2$ is negative or positive, respectively.

Now we will see the different possibilities.

\noindent\textbullet\kern8pt Case $c=1$. Then $M_\tau\subset\mathbb{M}_c^{n+1}=\mathbb{S}_1^{n+1}\subset\mathbb{R}_1^{n+2}$ and we have:
\begin{enumerate}
\item[i)] If $\<\va,\va\>=-1$, then $K=1/(\tau^2+1)$, $\e=-1$, and $M_\tau$ is isometric to a round sphere of radius $\sqrt{\tau^2+1}$, $M_\tau\equiv\mathbb{S}^n(\sqrt{\tau^2+1})$.
\item[ii)] If $\<\va,\va\>=0$, then $\tau\neq 0$, $K=0$, $\e=-1$, and $M_\tau$ is isometric to the Euclidean space, $M_\tau\equiv\mathbb{R}^n$.
\item[iii)] If $\<\va,\va\>=1$, then either $|\tau|>1$, $K=-1/(\tau^2-1)$, $\e=-1$, and $M_\tau$ is isometric to the hyperbolic space of radius $-\sqrt{\tau^2-1}$, $M_\tau\equiv\mathbb{H}^n(-\sqrt{\tau^2-1})$, or $|\tau|<1$, $K=1/(1-\tau^2)$, $\e=1$, and $M_\tau$ is isometric to a De Sitter space of radius $\sqrt{1-\tau^2}$, $M_\tau\equiv\mathbb{S}^n_1(\sqrt{1-\tau^2})$.
\end{enumerate}

\noindent\textbullet\kern8pt Case $c=-1$. Then $M_\tau\subset\mathbb{M}_c^{n+1}=\mathbb{H}_1^{n+1}\subset\mathbb{R}_2^{n+2}$ and we have:
\begin{enumerate}
\item[i)] If $\<\va,\va\>=-1$, then either $|\tau|>1$, $K=1/(\tau^2-1)$, $\e=1$, and $M_\tau$ is isometric to a De Sitter space of radius $\sqrt{\tau^2-1}$, $M_\tau\equiv\mathbb{S}^n_1(\sqrt{\tau^2-1})$, or $|\tau|<1$, $K=-1/(1-\tau^2)$, $\e=-1$, and $M_\tau$ is isometric to a hyperbolic space of radius $-\sqrt{1-\tau^2}$, $M_\tau\equiv\mathbb{H}^n(\sqrt{1-\tau^2})$.
\item[ii)] If $\<\va,\va\>=0$, then $\tau\neq0$, $K=0$, $\e=1$, and $M_\tau$ is isometric to the Lorentz-Minkowski space, $M_\tau\equiv\mathbb{R}^n_1$.
\item[iii)] If $\<\va,\va\>=1$, then $K=-1/(\tau^2+1)$, $\e=1$, and $M_\tau$ is isometric to the Lorentzian hyperbolic space, $M_\tau\equiv\mathbb{H}^n_1(-\sqrt{\tau^2+1})$.
\end{enumerate}

\end{ejem}

\begin{ejem}(Standard pseudo-Riemannian products in $\mathbb{M}^{n+1}_c$)\quad\label{ej3}
Let $f:\mathbb{M}_c^{n+1}\longrightarrow\mathbb{R}$ be the differentiable function defined by
\[
f(x)=cx_2^2+\delta_1\Big(-x^2_1+x_3^2+\ldots+x^2_{m+1}\Big)+\delta_2\Big(x^2_{m+2}+\ldots+x^2_{n+2}\Big),
\]
where $m\in\{1,\ldots,n\}$ and $\delta_1,\delta_2\in\{0,1\}$ with $\delta_1+\delta_2=1$. In short, $f(x)=\<Dx,x\>$, where $D$ is the matrix $D=\textrm{diag}[\delta_1,1,\delta_1\ldots,\delta_1,\delta_2,\ldots,\delta_2]$. Then, for every $r>0$ and $\rho=\pm1$ with $r^2- c\rho\neq0$,  the level set $M^n=f^{-1}(\rho r^2)$ is a hypersurface in $\mathbb{M}_c^{n+1}$, provided that $(\delta_1,\delta_2,\rho,c)\not\in\{(0,1,-1,1),(1,0,1,-1)\}$.

The Gauss map is given by
\begin{equation}\label{Nex3}
N(x)=\frac{\ol{\nabla} f(x)}{|\ol{\nabla} f(x)|}=\frac{1}{r\sqrt{\big|\rho-c\tau^2\big|}}\ (Dx-\rho cr^2x),
\end{equation}
and the shape operator is
\[
S=\frac{-1}{r\sqrt{\big|\rho-cr^2\big|}}\begin{bmatrix}
	(\delta_1-\rho cr^2)I_m\\
	&(\delta_2-\rho cr^2)I_{n-m}
\end{bmatrix}.
\]
In other words, $M^n$ has two principal curvatures
\[
\kappa_1=\dfrac{\rho cr^2-\delta_1}{r\sqrt{|\rho-cr^2|}}\quad\textrm{and}\quad
\kappa_2=\dfrac{\rho cr^2-\delta_2}{r\sqrt{|\rho-cr^2|}},
\]
with multiplicities $m$ and $n-m$, respectively. In particular, every mean curvature $H_k$ is constant. Therefore, by using~(\ref{E13}) and (\ref{Nex3}), we get that
\[
L_k\psi=c_kH_{k+1}N\circ \psi-cc_kH_k\psi
=\Big(\lambda\psi_1,\theta\psi_2,\lambda\psi_3,\ldots,\lambda\psi_m,\mu\psi_{m+1},\ldots,\mu\psi_{n+2}\Big),
\]
where
\[
\lambda=\frac{cc_kH_{k+1}(\delta_1-\rho cr^2)}{r\sqrt{\big|\rho-cr^2\big|}}-cc_kH_k,\qquad\theta=\frac{cc_kH_{k+1}(1-\rho cr^2)}{r\sqrt{\big|\rho-cr^2\big|}}-cc_kH_k,
\]
and
\[
\mu=\frac{cc_kH_{k+1}(\delta_2-\rho cr^2)}{r\sqrt{\big|\rho-cr^2\big|}}-cc_kH_k.
\]
That is, $M^n$ satisfies the condition $L_k\psi=\vA\psi+\vb$, with $\vb=0$ and  $\vA=\textrm{diag}[\lambda,\theta,\lambda,\ldots\lambda,\mu,\ldots,\mu]$.

The following two tables show the different hypersurfaces in $\mathbb{M}_c^{n+1}$.

\begin{center}
\begin{tabular}{|c|c|c|c|c|}
\hline
\multicolumn{5}{|c|}{Case $c=1$: Standard products in $\mathbb{S}_1^{n+1}$}\\\hline
$\delta_1$ & $\delta_2$ & $\rho$ & Hypersurface & $S$ \\\hline
1 & 0 & 1 & $\mathbb{S}^m_1(r)\times\mathbb{S}^{n-m}(\sqrt{1-r^2})$ &
    \rule[-16pt]{0pt}{40pt}
    $\begin{bmatrix}
    -\frac{\sqrt{1-r^2}}{r}I_m&\!\!0\\
    0&\!\!\frac{r}{\sqrt{1-r^2}}I_{n-m}
     \end{bmatrix}$  \\\hline
1 & 0 & $-1$ & $\mathbb{H}^m(-r)\times\mathbb{S}^{n-m}(\sqrt{1+r^2})$ &
    \rule[-16pt]{0pt}{40pt}
    $\begin{bmatrix}
    -\frac{\sqrt{1+r^2}}{r}I_m&\!\!0\\
    0&\!\!\frac{-r}{\sqrt{1+r^2}}I_{n-m}
     \end{bmatrix}$ \\\hline
0 & 1 & 1 & $\begin{array}{c}
    \mathbb{S}^m_1(\sqrt{1-r^2})\times\mathbb{S}^{n-m}(r)\\
    \mathbb{H}^m(-\sqrt{r^2-1})\times\mathbb{S}^{n-m}(r)
    \end{array}$  &
    \rule[-20pt]{0pt}{45pt}
    $\begin{bmatrix}
    \frac{r}{\sqrt{|r^2-1|}}I_m&\!\!0\\
    0&\!\!\frac{r^2-1}{r\sqrt{|r^2-1|}}I_{n-m}
     \end{bmatrix}$ \\\hline
\end{tabular}
\end{center}

\begin{center}
\begin{tabular}{|c|c|c|c|c|}
\hline
\multicolumn{5}{|c|}{Case $c=-1$: Standard products in $\mathbb{H}_1^{n+1}$}\\\hline
$\delta_1$ & $\delta_2$ & $\rho$ & Hypersurface & $S$ \\\hline
1 & 0 & $-1$ & $\mathbb{H}^m_1(-r)\times\mathbb{S}^{n-m}(\sqrt{r^2-1})$ &
    \rule[-16pt]{0pt}{40pt}
    $\begin{bmatrix}
    \frac{\sqrt{r^2-1}}{r}I_m&\!\!0\\
    0&\!\!\frac{r}{\sqrt{r^2-1}}I_{n-m}
     \end{bmatrix}$  \\\hline
0 & 1 & $1$ & $\mathbb{H}^m(-\sqrt{1+r^2})\times\mathbb{S}_1^{n-m}(r)$ &
    \rule[-16pt]{0pt}{40pt}
    $\begin{bmatrix}
    \frac{-r}{\sqrt{1+r^2}}I_m&\!\!0\\
    0&\!\!-\frac{\sqrt{1+r^2}}{r}I_{n-m}
     \end{bmatrix}$ \\\hline
0 & 1 & $-1$ & $\begin{array}{c}
    \mathbb{S}^m_1(\sqrt{r^2-1})\times\mathbb{H}^{n-m}(-r)\\
    \mathbb{H}^m(-\sqrt{1-r^2})\times\mathbb{H}^{n-m}(-r)
    \end{array}$  &
    \rule[-20pt]{0pt}{45pt}
    $\begin{bmatrix}
    \frac{r}{\sqrt{|r^2-1|}}I_m&\!\!0\\
    0&\!\!\frac{r^2-1}{r\sqrt{|r^2-1|}}I_{n-m}
     \end{bmatrix}$ \\\hline
\end{tabular}
\end{center}

\end{ejem}

\begin{ejem}(A quadratic hypersurface with non-diagonalizable shape operator)\label{ej4}\quad
The hypersurfaces shown in Examples \ref{ej2} and \ref{ej3} have diagonalizable shape operators. However, since we are working in a Lorentzian space form, it seems natural thinking of hypersurfaces with non-diagonalizable shape operator satisfying $L_k\psi=A\psi+b$. Let $R$ be a self-adjoint endomorphism of $\mathbb{R}^{n+2}_q$, that is, $\left< Rx,y\right> =\left< x,Ry\right>$, for all $x,y\in\mathbb{R}^{n+2}_q$. Let $f:\mb\fle\R{}$ be a quadratic function defined by $f(x)=\left< Rx,x\right>$, and assume that the minimal polynomial of $R$ is given by $\mu_R(t)=t^2+at+b$, $a,b\in\R{}$, with $a^2-4b\leq0$. Then, by computing the gradient in $\mb$ at each point $x\in\mb$, we have $\ol\nabla f(x)=2Rx-2cf(x)x$.

Let us consider the level set $M=f^{-1}(d)$, for a real constant $d$. Then, at a point $x$ in $M$, we have
\[
\left<\ol\nabla f(x),\ol\nabla f(x)\right>=4\left< R^2x,x\right> -4cf(x)^2=-4c\mu_R(cd),
\]
where we have used that $R^2x=-aRx-bx$. Then, for every $d\in\R{}$ with $\mu_R(cd)\neq 0$, $M=f^{-1}(d)$ is a Lorentzian hypersurface in $\mb$. The Gauss map at a point $x$ is given by
\begin{equation}\label{B1.ej4}
N(x)=\frac{1}{|\mu_R(cd)|^{1/2}}\ (Rx-cd x),
\end{equation}
and thus the shape operator is given by
\begin{equation}\label{S.ej4}
SX=-\frac{1}{|\mu_R(cd)|^{1/2}}(RX-cdX),
\end{equation}
for every tangent vector field $X$. From here, and bearing in mind that $R^2+aR+bI=0$, we obtain that
\[
S^2X=-\frac{1}{|\mu_R(cd)|}\left((a+2cd)RX+(b-d^2)X\right),
\]
for every tangent vector field $X$. At this point, it is very easy to deduce that
\[
\mu_S(t)=t^2-\frac{a+2cd}{|\mu_R(cd)|^{1/2}}t+\frac{b+d^2+acd}{|\mu_R(cd)|}
\]
is the minimal polynomial of $S$, and that every $k$-th mean curvature is constant. On the other hand, since the discriminant of $\mu_S(t)$ is not positive, the shape operator is non-diagonalizable.

Finally, from (\ref{E13}), we obtain that $L_k\psi=A\psi$, where $A$ is the matrix given by
\[
A=\frac{c_kH_{k+1}}{|\mu_R(cd)|^{1/2}}R-\left(\frac{c_kH_{k+1}cd}{|\mu_R(cd)|^{1/2}}+cc_kH_k\right)I.
\]
\end{ejem}

\section{First results}
\label{s:first.results}

In this section we need to compute $L_kN$, and to do that we are going to compute the operator $L_k$ acting on the coordinate functions of the Gauss map $N$, that is, the functions $\<a,N\>$ where $a\in\mathbb{R}^{n+2}_q$ is an arbitrary fixed vector. A straightforward computation yields
\begin{equation*}
\nabla\< \va,N\>=-S\va^\top.
 \end{equation*}
From Weingarten formula and~(\ref{E11}), we find that
\begin{align*}
\nabla_X\nabla\<\va,N\>&=- \nabla_X(S\va^\top)=-(\nabla_XS)\va^\top -S(\nabla_X\va^\top)\\
&=-(\nabla_{\va^\top}S)X-\varepsilon\<\va,N\>S^2X+c\<\va,\psi\>SX,
\end{align*}
for every tangent vector field $X$. This equation, jointly with Lemma~\ref{L1} and~(\ref{E9}), yields
\begin{align}\label{E15}
L_k\<\va,N\>
    &=-\tr(P_{k}\circ\nabla_{\va^\top}S)-\e\<\va,N\>\tr(P_{k}\circ S^2) +c\<\va,\psi\>\tr(P_{k}\circ S)\nonumber\\
	&=-\e C_k\<\nabla H_{k+1},\va^\top\>-\e C_k(nH_1H_{k+1}-(n-k-1)H_{k+2})\<\va,N\>\nonumber\\
	&\quad+\e cc_kH_{k+1}\<\va,\psi\>.
\end{align}
In other words,
\begin{equation}\label{E16}
L_kN=-\varepsilon C_k\nabla H_{k+1}-\varepsilon C_k\Big(nH_1H_{k+1}-(n-k-1)H_{k+2}\Big)N+\varepsilon cc_kH_{k+1}\psi.
\end{equation}
On the other hand, equations~(\ref{Lkfg}) and (\ref{E12}) lead to
\begin{align*}
L_k(L_k\<\va,\psi\>)&=c_kH_{k+1}L_k\<\va,N\>+L_k(c_kH_{k+1})\<\va,N\>+2c_k\big\langle P_k(\nabla H_{k+1}),\nabla\<\va,N\>\big\rangle\\
&\quad - cc_kH_{k}L_k\<\va,\psi\>- L_k(cc_kH_{k})\<\va,\psi\>-2cc_k\big\langle P_k(\nabla H_{k}),\nabla\<\va,\psi\>\big\rangle,
\end{align*}
and by using again~(\ref{E12}) and~(\ref{E15}) we get that
\begin{align*}
L_k\big(L_k\<\va,\psi\>\big)&=-\varepsilon c_kC_kH_{k+1}\big\langle \nabla H_{k+1},\va\big\rangle-2c_k\big\langle(S\circ P_k)(\nabla H_{k+1}),\va\big\rangle-2cc_k\big\langle P_k(\nabla H_{k}),\va\big\rangle \\
&\quad -\big[\e C_k H_{k+1}\big(nH_1H_{k+1}-(n-k-1)H_{k+2}\big)+cc_kH_kH_{k+1}-L_k(H_{k+1})\big]c_k\big\langle \va,N\big\rangle\\
&\quad +\big[\varepsilon cc_kH^2_{k+1}+c_kH_k^2-cL_k(H_k)\big]c_k\< \va,\psi\>.
\end{align*}
Therefore, we get
\begin{align}\label{J1}
L_k\big(L_k\psi\big)&=-\varepsilon c_kC_kH_{k+1}\nabla H_{k+1}-2c_k(S\circ P_k)(\nabla H_{k+1})-2cc_k P_k(\nabla H_{k})\nonumber \\
&\quad -\big[\e C_k H_{k+1}\big(nH_1H_{k+1}-(n-k-1)H_{k+2}\big)+cc_kH_kH_{k+1}-L_k(H_{k+1})\big]c_kN\nonumber\\
&\quad +\big[\e cc_kH^2_{k+1}+c_kH_k^2-cL_k(H_k)\big]c_k\psi.
\end{align}
\indent Let us assume that, for a fixed $k=0,1,\ldots,n-1$, the immersion $\psi:M^n\longrightarrow\mathbb{M}^{n+1}_c$ satisfies the condition
\begin{align}\label{E17}
L_k\psi&=\vA\psi+\vb,
\end{align}
for a constant matrix $\vA\in\mathbb{R}^{(n+2)\times(n+2)}$ and a constant vector $\vb\in\mathbb{R}_q^{n+2}$.
Then we have $L_k\big(L_k\psi\big)=\vA L_k\psi$, that, jointly with~(\ref{J1}) and~(\ref{E13}), yields
\begin{align}\label{E19}
H_{k+1}\vA N-cH_k\vA\psi&=-\e C_kH_{k+1}\nabla H_{k+1}-2(S\circ P_k)(\nabla H_{k+1})-2c P_k(\nabla H_{k})\nonumber\\
& -\big[\varepsilon C_k H_{k+1}\big(nH_1H_{k+1}-(n-k-1)H_{k+2}\big)+cc_kH_kH_{k+1}-L_k(H_{k+1})\big]N\nonumber\\ &+\big[\varepsilon cc_kH^2_{k+1}+c_kH_k^2-cL_k(H_k)\big]\psi.
\end{align}
On the other hand, from (\ref{E17}), and using again (\ref{E13}), we have
\begin{align}\label{E17A}
\vA\psi&=c_kH_{k+1} N-cc_kH_k\psi-\vb^\top-\varepsilon\< \vb,N\> N-c\< \vb,\psi \>\psi\nonumber\\
&=-\vb^\top+\big[c_kH_{k+1}-\varepsilon\< \vb,N\>\big]N-\big[cc_kH_k+c\< \vb,\psi \>\big]\psi,
\end{align}
where $b^\top\in\mathfrak{X}(M)$ denotes the tangential component of $b$. Finally, from here and (\ref{E19}), we get
\begin{align}\label{E19A}
H_{k+1}\vA N&=-\e C_kH_{k+1}\nabla H_{k+1}-2(S\circ P_k)(\nabla H_{k+1})-2c P_k(\nabla H_{k})-cH_k\vb^\top\nonumber\\
&\quad-\big[\e C_k H_{k+1}\big(nH_1H_{k+1}-(n-k-1)H_{k+2}\big)+\e cH_k\< \vb,N\> -L_k(H_{k+1})\big]N\nonumber\\
&\quad+\big[\e cc_kH^2_{k+1}-H_k\< \vb,\psi\>- cL_k(H_k)\big]\psi.
\end{align}

\subsection{The case where $A$ is self-adjoint}
If we take covariant derivative in (\ref{E17}), and use equation (\ref{E13}) as well as Weingarten formula, we have
\begin{equation}\label{E18}
	\vA X=-c_kH_{k+1}SX-cc_kH_kX+c_k\<\nabla H_{k+1},X\> N-cc_k\<\nabla H_{k},X\>\psi,
\end{equation}
for every tangent vector field $X$, and therefore
\[
\<\vA X,Y\>=\< X,\vA Y\>,
\]
for every tangent vector fields $X,Y\in\mathfrak{X}(M)$. Therefore, $\vA$ is self-adjoint if and only if the following conditions hold
\begin{align}
\<\vA X,\psi\>&=\<X,\vA\psi\>,\label{21}\\
\<\vA X,N\>&=\< X,\vA N\>,\label{22}\\
\< \vA N,\psi\>&=\< N,\vA\psi\>,\label{23}
\end{align}
for every vector field $X\in\mathfrak{X}(M)$. From (\ref{E18}) and (\ref{E17A}), we easily see that (\ref{21}) is equivalent to
\begin{equation}\label{24}
\nabla\<\vb,\psi\>=\vb^\top=c_k\nabla H_{k},
\end{equation}
and so $\<\vb,\psi\>-c_kH_{k}$ is constant on $M$. A direct consequence is that
\begin{align}\label{26}
c_kL_k(H_k)=L_k\< \vb,\psi\>=\big[c_kH_{k+1}\< \vb,N\>-cc_kH_k\<\vb,\psi\>\big],
\end{align}
that, jointly with (\ref{E17A}) and (\ref{E19}), yields
\begin{align*}
H_{k+1}\< \vA N,\psi\>&=\e c_kH^2_{k+1}+cc_kH_k^2-L_k(H_k) +cH_k\< \vA\psi,\psi\>\\
&=\e c_kH^2_{k+1}+cc_kH_k^2-\big[H_{k+1}\< \vb,N\>-cH_k\< \vb,\psi\>\big] +cH_k\big[-c_kH_{k}-\< \vb,\psi\>\big]\\	
&=\e c_kH^2_{k+1}-H_{k+1}\< \vb,N\>\\
&=H_{k+1}\< N,\vA\psi\>.
\end{align*}
Therefore, at points where $H_{k+1}\neq0$, equation (\ref{21}) implies equation (\ref{23}). Even more, if $H_{k+1}\neq0$ then from (\ref{E18}), (\ref{E19A}) and (\ref{24}), it is easy to see that equation (\ref{22}) is equivalent to
\begin{align}\label{25}
\frac{2}{H_{k+1}}(S\circ P_k)(\nabla H_{k+1})+\varepsilon(k+2) C_k\nabla H_{k+1}=-\frac{c}{H_{k+1}}\big(2P_k(\nabla H_k)+c_kH_k\nabla H_{k}\big).	
\end{align}
The following auxiliar result is the key point in the proof of the main theorems.

\begin{lema}\label{L7}
Let $\psi:M\longrightarrow \mathbb{M}^{n+1}_c\subset\mathbb{R}^{n+2}_q$ be an orientable hypersurface satisfying the condition $L_k\psi=\vA\psi+\vb$, for a fixed $k=0,1,\ldots,n-1$, some self-adjoint constant matrix $\vA\in\mathbb{R}^{(n+2)\times(n+2)}$ and some constant vector $\vb\in\mathbb{R}^{n+2}_q$. Then $H_{k}$ is constant if and only if $H_{k+1}$ is constant.
\end{lema}
\begin{proof}
Let us assume that $H_k$ is constant, and consider the open set
\[
\U=\{p\in M\;|\;\nabla H^2_{k+1}(p)\neq 0\}.
\]
Our goal is to show that $\U$ is empty. If $\U$ is not empty then, from (\ref{25}), we have that
\begin{equation}\label{E23}
(S\circ P_k)(\nabla H_{k+1})=-\frac{\varepsilon(k+2)C_k}{2}H_{k+1}\nabla H_{k+1}\quad\textrm{on } \U.	
 \end{equation}
Then reasoning exactly as Lucas and Ramírez in \cite[Lemma 9]{LR2010a} (starting from equation (26) in \cite{LR2010a}) we conclude that $H_{k+1}$ is locally constant on $\U$, which is not possible. The proof in \cite{LR2010a} also works here word by word, with the only difference that here $H_k$ is constant, and then (\ref{E18}) reduces now to
\[
	\vA X=-c_kH_{k+1}SX-cc_kH_kX+c_k\< \nabla H_{k+1},X\> N.
\]
Therefore, now we have $AE_i=-c_k\big(H_{k+1}\kappa_i+cH_k\big)E_i$, for $m+1\leq i\leq n$ (see the last part of the proof of \cite[Lemma 9]{LR2010a}). Since $H_k$ is constant, that makes no difference to the reasoning.

Conversely, let us assume that $H_{k+1}$ is constant, and suppose that the open set
\[
\V=\{p\in M\;|\;\nabla H^2_{k}(p)\neq 0\}
\]
is non-empty. First, let us consider the case where $H_{k+1}=0$. Then, from (\ref{24}) and (\ref{26}), equation (\ref{E19A}) reduces to
\[
-2cP_k(\nabla H_k)-cc_kH_k\nabla H_k-\e cH_k\< \vb,N\> N=0,
\]
and so $\<\vb,N\>=0$ on $\V$. From (\ref{E17A}), we have $\<\vA N,\psi\>=\<N,\vA\psi\>=0$ and $\<\vA N,X\>=\< N,\vA X\>=0$, and then $AN=\lambda N$, i.e., $N$ is an eigenvector of $A$ with corresponding eigenvalue $\lambda=\e\<\vA N,N\>$. In particular, $\lambda$ is locally constant on $\V$. Therefore,
\begin{align*}
\vA X&=-cc_kH_kX-cc_k\< \nabla H_{k},X\>\psi,\\
\vA N&=\lambda N,\\
\vA\psi&=-\vb^\top-(cc_kH_k+c\<\vb,\psi\>)\psi=-c_k\nabla H_k-(2cc_kH_k+\alpha)\psi,
\end{align*}
where $\alpha=c\<\vb,\psi\>-cc_kH_k$ is also locally constant on $\V$. Then since $\tr(\vA)=-ncc_kH_k+\lambda-2cc_kH_k-\alpha$ is constant, this implies that $H_k$ is locally constant on $\V$, which is a contradiction.

Let us consider now that $H_{k+1}$ is a non-zero constant. Then, from (\ref{25}), we get
\begin{equation}\label{E26A}
P_k(\nabla H_k)+D_kH_k\nabla H_{k}=0\quad \textrm{on } \V,
\end{equation}
where $D_k=c_k/2$. From now on, we will follow a similar reasoning to that given in \cite[Lemma 9]{LR2010a}. The proof continues according to the type of the shape operator $S$.

\paragraph{Case 1: \emph{S} is of type I.}
Consider $\{E_1,E_2,\ldots,E_n\}$ a local orthonormal frame of principal directions of $S$ (see Proposition~\ref{caso1}). The vector field $\nabla H_{k}$ can be written as
\[
\nabla H_{k}=\sum^n_{i=1}\e_i\<\nabla H_{k},E_i\>E_i,\quad\e_i=\<E_i,E_i\>,
\]
and thus we get
\[
P_{k}(\nabla H_{k})=\sum^n_{i=1}\e_i\<\nabla H_{k},E_i\>\Big((-1)^{k}\mu_{_{k}}^i\Big)E_i.
\]
Then equation (\ref{E26A}) is equivalent to
\[
\<\nabla H_{k},E_i\>\bigg(D_kH_{k}+(-1)^{k}\mu_{_{k}}^i\bigg)=0\qquad\textrm{on } \V,
\]
for every $i\in\{1,\ldots,n\}$. Therefore, for each $i$ such that $\<\nabla H_{k},E_i\>\neq0$ on $\V$, we have
\begin{equation}\label{E27}
(-1)^{k}\mu_{_{k}}^i=-D_k\,H_{k}.
\end{equation}
We claim that $\<\nabla H_{k},E_i\>=0$ for some $i$. Otherwise, (\ref{E27}) holds for every $i$, which implies
\[
\tr(P_{k})=\sum_{i=1}^n\e_i\< P_{k}E_i,E_i\>=\sum_{i=i}^n(-1)^{k}\mu_{_{k}}^i=-nD_kH_{k}.
\]
But then, from Lemma~\ref{Pro}, we obtain that $H_{k}=0$ on $\V$, which is not possible.

Now re-arranging the local orthonormal frame if necessary (or even taking another orthonormal frame of principal directions), we may assume that there exists some $m\in\{1,\ldots,n-1\}$ such that
\begin{align}\label{I2}
\<\nabla H_{k},E_i\>\neq0& \textrm{ for } i=1,\ldots,m,\text{ and }\kappa_1<\cdots<\kappa_m.\\
\<\nabla H_{k},E_i\>=0& \textrm{ for } i=m+1,\ldots,n.\nonumber
\end{align}
\begin{claim}\label{claim1}
For every subset $J\subseteq\{1,\ldots,m\}$ we have
\begin{align}\label{E29}
\mu_{_{k}}^J=(-1)^{k+1}D_kH_{k}.
\end{align}
\end{claim}
We will prove (\ref{E29}) by induction on the cardinality $\card(J)$ of the set $J$. For $\card(J)=1$,  equation~(\ref{E29}) is nothing but (\ref{E27}). Let us assume that (\ref{E29}) holds for every set $J$ with $\card(J)=1,2,\ldots,p<m$, and take a set $J_0=\big\{j_{1},\ldots,j_{p+1}\big\}\subseteq\{1,\dots,m\}$ with cardinality $p+1\leq m$. Let $J_1$ and $J_2$ be the two sets of cardinality $p$ such that
\begin{align*}
J_0=\underbrace{\big\{j_{1},j_3,\ldots,j_{p+1}\big\}}_{J_2}\cup\big\{j_{2}\big\}= \underbrace{\big\{j_{2},j_3,\ldots,j_{p+1}\big\}}_{J_1}\cup\big\{j_{1}\big\}.
\end{align*}
By using the induction hypothesis applied to $J_1$ and $J_2$, we have
$\mu_{_{k}}^{J_1}=\mu_{_{k}}^{J_2}=	(-1)^{k+1}D_kH_{k}$.
Now, bearing in mind~(\ref{F}), from the first equality of last equation we obtain
\[
\kappa_{j_2}\mu_{_{k-1}}^{J_0}+\mu_{_{k}}^{J_0}=\kappa_{j_1}\mu_{_{k-1}}^{J_0}+\mu_{_{k}}^{J_0},
\]
and by using~(\ref{I2}) we get $\mu_{_{k-1}}^{J_0}=0$, and so $\mu_{_{k}}^{J_1}=\mu_{_{k}}^{J_2}=\mu_{_{k}}^{J_0}$. That concludes the proof of the Claim \ref{claim1}.

Finally, from (\ref{E18}) and (\ref{I2}) we have $AE_i=\eta_iE_i$, $i=m+1,\ldots,n$, where $\eta_i=-c_kH_{k+1}\kappa_i-cc_kH_k$ is a constant eigenvalue of the constant matrix $\vA$. On the other hand, from (\ref{E29}) for the set $J=\{1,\ldots,m\}$, we have
\[
(-1)^{k+1}D_kH_{k}=\sum_{m<i_1<\cdots<i_{k}}\!\!\!\!\!\!\!\kappa_{i_1}\cdots\kappa_{i_{k}}=\frac{\displaystyle\sum_{m<i_1<\cdots<i_{k}}\!\!\!\!\!\!\!\!\big(\eta_{i_1}+cc_kH_k\big)\cdots \big(\eta_{i_{k}}+cc_kH_k\big)}{\big(-c_kH_{k+1}\big)^{k}}.
\]
In other words,
\[
(-1)^{k+1}D_kH_{k}=B_0+B_1H_k+\cdots+B_kH_k^k,
\]
for certain constants $B_i$. Therefore, $H_{k}$ is locally constant on $\V$, which is a contradiction. This finishes the proof in the Case 1.

\paragraph{Case 2: \emph{S} is of type II.}
Let $\{E_1,E_2,\ldots,E_n\}$ be an orthonormal frame giving the canonical form of $S$ (see Proposition~\ref{caso2}). The gradient $\nabla H_{k}$ can be written in this basis as follows
\[
\nabla H_{k}=-\< \nabla H_{k},E_1\> E_1+\< \nabla H_{k},E_2\> E_2+\sum^n_{i=3}\< \nabla H_{k},E_i\> E_i,
\]
and then we get
\begin{align*}
P_{k}(\nabla H_{k})&=-(-1)^{k}\<\nabla H_{k},E_1\>\Big(\mu_{_{k}}^1E_1+b\mu_{_{k-1}}^{1,2}E_2\Big)\\[5pt]
&\quad+(-1)^{k}\<\nabla H_{k},E_2\>\Big(\!-b\mu_{_{k-1}}^{1,2}E_1+\mu_{_{k}}^1E_2\Big)\\
&\quad+(-1)^{k}\sum^n_{i=3}\<\nabla H_{k},E_i\>\Big(\mu_{_{k}}^i+b^2\mu_{_{k-2}}^{1,2,i}\Big)E_i.
\end{align*}
Now, bearing in mind~(\ref{E26A}), we obtain the following equations on $\V$:
\begin{align*}
&\<\nabla H_{k},E_1\>\Big(D_kH_{k}+(-1)^{k}\mu_{_{k}}^1\Big)+\<\nabla H_{k},E_2\>(-1)^{k}b\mu_{_{k-1}}^{1,2}=0,\\[5pt]
&\<\nabla H_{k},E_2\>\Big(D_kH_{k}+(-1)^{k}\mu_{_{k}}^1\Big)-\<\nabla H_{k},E_1\>(-1)^{k}b\mu_{_{k-1}}^{1,2}=0,\\
&\<\nabla H_{k},E_i\>\Big(D_kH_{k}+(-1)^{k}(\mu_{_{k}}^i+b^2\mu_{_{k-2}}^{1,2,i})\Big)=0,\qquad\textrm{for } i=3,\ldots,n.
\end{align*}
Therefore, if $\<\nabla H_{k},E_1\>\neq0$ or $\<\nabla H_{k},E_2\>\neq 0$ then we have
\begin{equation}\label{2E27A}
(-1)^{k}\mu_{_{k}}^1=-D_kH_{k}\quad\textrm{and}\quad\mu_{_{k-1}}^{1,2}=0.	
\end{equation}
Moreover, for every $i=3,\dots,n$, such that $\<\nabla H_{k},E_i\>\neq0$, we have
\begin{equation}\label{2E27}
(-1)^{k}(\mu_{_{k}}^i+b^2\mu_{_{k-2}}^{1,2,i})=-D_kH_{k}.
\end{equation}
We claim that $\<\nabla H_{k},E_i\>=0$ for some $i$. Otherwise, (\ref{2E27A}) holds and (\ref{2E27}) is true for every $i\geq 3$. Thus, we deduce
\begin{align*}
\tr(P_{k})&=-\<P_{k}E_1,E_1\>+\<P_{k}E_2,E_2\>+\sum^n_{i=3}\<P_{k}E_i,E_i\>\\
&=(-1)^{k}\mu_{_{k}}^1+(-1)^{k}\mu_{_{k}}^1+\sum_{i=3}^n(-1)^{k}\big(\mu_{_{k}}^i+b^2\mu_{_{k-2}}^{1,2,i}\big)=-nD_kH_{k}.
\end{align*}
But this means, from Lemma~\ref{L1}\emph{(b)}, that $H_{k}=0$ on $\V$, which is a contradiction.

Observe that when $\<\nabla H_{k},E_i\>\neq0$ for some $i\geq 3$, then (after re-arranging the local orthonormal frame if necessary) we may assume that there exists some $m\in\{3,\ldots,n\}$ such that
\begin{align}\label{2I2}
\<\nabla H_{k},E_i\>\neq0& \textrm{ for } i=3,\ldots,m,\text{ and }\kappa_3<\cdots<\kappa_m.\\
\<\nabla H_{k},E_i\>=0& \textrm{ for } i=m+1,\ldots,n.\nonumber
\end{align}

\begin{claim}\label{claim2}
If $\<\nabla H_{k},E_1\>=\<\nabla H_{k},E_2\>=0$, then for every non-empty subset $J\subseteq\{3,\ldots,m\}$ we have
\begin{align}\label{2E29}
\mu_{_{k}}^J+b^2\mu_{_{k-2}}^{1,2,J}=(-1)^{k+1}D_kH_{k},
\end{align}
where $m\in\{3,\ldots,n\}$ is the number such that (\ref{2I2}) holds.
\end{claim}
We will show~(\ref{2E29}) by induction on the cardinality of set $J$, $\card(J)$. If $\card(J)=1$,  then~(\ref{2E29}) is nothing but~(\ref{2E27}). Let us assume that~(\ref{2E29}) holds for subsets $J$ with $\card(J)=1,2,\ldots,p<m-2$, and take a set $J_0=\big\{j_{1},\ldots,j_{p+1}\big\}\subseteq\big\{3,\ldots,m\big\}$ with cardinality $p+1\leq m-2$. Let $J_1$ and $J_2$ be the sets of cardinality $p$ such that
\begin{align*}
J_0&=\underbrace{\big\{j_{1},j_3,\ldots,j_{p+1}\big\}}_{J_2}\cup\big\{j_{2}\big\}
=\underbrace{\big\{j_{2},j_3,\ldots,j_{p+1}\big\}}_{J_1}\cup\big\{j_{1}\big\}.
\end{align*}
By the induction hypothesis applied to $J_1$ and $J_2$ we have
\[
\mu_{_{k}}^{J_1}+b^2\mu_{_{k-2}}^{1,2,J_1}=
\mu_{_{k}}^{J_2}+b^2\mu_{_{k-2}}^{1,2,J_2}=(-1)^{k+1}D_kH_{k},
\]
and then, by using~(\ref{F}) in the first equality, we get
\[
(\kappa_{j_2}-\kappa_{j_1})\big(\mu_{_{k-1}}^{J_0}+b^2\mu_{_{k-3}}^{1,2,J_0}\big)=0.
\]
But $\kappa_{j_2}\neq\kappa_{j_1}$ (see~(\ref{2I2})), and so $\mu_{_{k-1}}^{J_0}+b^2\mu_{_{k-3}}^{1,2,J_0}=0$. This yields
\[
\mu_{_{k}}^{J_1}+b^2\mu_{_{k-2}}^{1,2,J_1}=
\mu_{_{k}}^{J_2}+b^2\mu_{_{k-2}}^{1,2,J_2}=
\mu_{_{k}}^{J_0}+b^2\mu_{_{k-2}}^{1,2,J_0},
\]
and the proof of the Claim~\ref{claim2} finishes.

\begin{claim}\label{claim3}
If $\<\nabla H_{k},E_1\>\neq0$ or $\<\nabla H_{k},E_2\>\neq0$, then for every $J\subseteq\{3,\ldots,m\}$ (admitting $J=\emptyset$) we have
\begin{align}
\text{a) } \mu_{_{k-1}}^{1,2,J}&=0,\label{2E29A2}\\
\text{b) } \mu_{_{k}}^{1,2,J}&=(-1)^{k+1}D_kH_{k},\label{2E29A}
\end{align}
where $m\in\{3,\dots,m\}$ is the number such that (\ref{2I2}) holds.
\end{claim}
[We note here that if there is no number $m\geq 3$ such that (\ref{2I2}) holds, then this claim only refers to $J=\emptyset$.]
First, we prove (a) by induction on $\card(J)$. If $\card(J)=0$, then (\ref{2E29A2}) is nothing but the second equation of (\ref{2E27A}). Let us assume that (\ref{2E29A2}) holds for subsets $J$ with $\card(J)=0,1,\ldots,p<m-2$, and take a set $J_0=\big\{j_{1},\ldots,j_{p+1}\big\}\subseteq\{3,\ldots,m\}$ with cardinality $p+1\leq m-2$. Let $J_1$ and $J_2$ be the sets of cardinality $p$ such that
\begin{align*}
J_0=&\underbrace{\big\{j_{1},\ldots,j_{p}\big\}}_{J_{1}}\cup\big\{j_{p+1}\big\}=
\underbrace{\big\{j_{1},\ldots,j_{p-1},j_{p+1}\big\}}_{J_2}\cup\big\{j_{p}\big\}.
\end{align*}
By the induction hypothesis applied to $J_1$ and $J_2$, we have
$\mu_{_{k-1}}^{1,2,J_1}=\mu_{_{k-1}}^{1,2,J_2}=0$.
Now, by using (\ref{F}), we get $(\kappa_{j_{p+1}}-\kappa_{j_p})\mu_{_{k-2}}^{1,2,J_0}=0$, and from (\ref{2I2}) we obtain $\mu_{_{k-2}}^{1,2,J_0}=0$. That leads to $\mu_{_{k-1}}^{1,2,J_0}=\mu_{_{k-1}}^{1,2,J_1}=\mu_{_{k-1}}^{1,2,J_2}$, and the proof of (a) finishes.

The proof of (b) is similar and is also made by induction on $\card(J)$. If $\card(J)=0$, since $\mu^1_{_{k}}=\mu^{1,2}_{_{k}}+\kappa\mu^{1,2}_{_{k-1}}$, then (\ref{2E29A}) follows by using the first equation of (\ref{2E27A}) and the claim (a). Let us assume that (\ref{2E29A}) holds for subsets $J$ with $\card(J)=0,1,\ldots,p<m-2$ and take a set $J_0=\big\{j_{1},\ldots,j_{p+1}\big\}\subseteq\{3,\ldots,m\}$ with cardinality $p+1\leq m-2$. Let $J_1$ be the set of cardinality $p$ such that
\[
J_0=\underbrace{\big\{j_{1},\ldots,j_{p}\big\}}_{J_{1}}\cup\big\{j_{p+1}\big\}.
\]
By the induction hypothesis applied to $J_1$, and bearing in mind (\ref{2E29A2}), we have $\mu_{_{k}}^{1,2,J_1}=\kappa_{p+1}\mu_{_{k-1}}^{1,2,J_0}+\mu_{_{k}}^{1,2,J_0}=\mu_{_{k}}^{1,2,J_0}$, and this finishes the proof of Claim~\ref{claim3}.

Observe that if $J\neq\emptyset$, then equations (\ref{2E29}) and (\ref{2E29A}) lead to
\[
\mu_{_{k}}^{1,2,J}=\mu_{_{k}}^{J}+b^2\mu_{_{k-2}}^{1,2,J}=(-1)^{k+1}D_kH_{k}.	
\]
Now, we can use (\ref{F}) to get $\mu_{_{k}}^{J}=\mu_{_{k}}^{1,2,J}+2\kappa\mu_{_{k-1}}^{1,2,J}+ \kappa^2\mu_{_{k-2}}^{1,2,J}$. Putting this into last equation we obtain
$2\kappa\mu_{_{k-1}}^{1,2,J}+(\kappa^2+b^2)\mu_{_{k-2}}^{1,2,J}=0$,
and as a consequence of Claim \ref{claim3}(a) we deduce
$\mu_{_{k-2}}^{1,2,J}=0$, for every non-empty set $J\subseteq\{3,\dots,m\}$.

Finally, from (\ref{E18}) and (\ref{2I2}), we have $AE_i=\eta_iE_i$, for $i=m+1,\ldots,n$, where $\eta_i=-c_kH_{k+1}\kappa_i-cc_kH_k$ is a constant eigenvalue of the constant matrix $A$. As a consequence we deduce
\begin{align}\label{p2}
\mu^{1,\ldots,m}_{_{r}}=\sum_{m<i_1<\cdots<i_{r}}\!\!\!\!\!\!\!\kappa_{i_1}\cdots\kappa_{i_{r}}&=\frac{\displaystyle\sum_{m<i_1<\cdots<i_{r}}\!\!\!\!\!\!\!\!\big(\eta_{i_1}+cc_kH_k\big)\cdots \big(\eta_{i_{r}}+cc_kH_k\big)}{\big(-c_kH_{k+1}\big)^{r}}\nonumber\\[5pt]
&=B_0+B_1H_k+\cdots+B_rH_k^r,
\end{align}
for certain constants $B_i$. To finish the proof in Case 2, we distinguish two subcases.

\textbf{(2.1)} If $\<\nabla H_{k},E_1\>=\<\nabla H_{k},E_2\>=0$, let $V$ be the plane spanned by $\{E_1,E_2\}$. Since we have
\begin{align*}
AE_1 &= -c_k(H_{k+1}\kappa+cH_k)E_1-c_kH_{k+1}bE_2,\\
AE_2 &= c_kH_{k+1}bE_1-c_k(H_{k+1}\kappa+cH_k)E_2,
\end{align*}
then $V$ is an invariant subspace, and thus the operator $A|_V$ has constant invariants
$\theta=\tr(A|_V)$ and $\beta=\det(A|_V)$, which are given by
\begin{align*}
\theta&=-2c_k(H_{k+1}\kappa+cH_k),\\
\beta&=c_k^2(H_{k+1}\kappa+cH_k)^2+c_k^2H_{k+1}^2b^2.
\end{align*}
Thus, we can find constants $\theta_i$, $B_i$ such that
\begin{align}\label{V1}
2\kappa &= \theta_0+\theta_1H_k,\nonumber\\
\kappa^2+b^2 &= B_0+B_1H_k+B_2H_k^2.
\end{align}
On the other hand, by using (\ref{2E29}) for $J=\{3,\dots,m\}$, we get
\begin{align*}
(-1)^{k+1}D_kH_{k}&=\mu_{_{k}}^{3,\ldots,m}+b^2\mu_{_{k-2}}^{1,\ldots,m}\\
&=\mu_{_{k}}^{1,\ldots,m}+2\kappa\mu_{_{k-1}}^{1,\ldots,m}+(\kappa^2+b^2)\mu_{_{k-2}}^{1,\ldots,m},	
\end{align*}
that, jointly with (\ref{p2}) and (\ref{V1}), yields
\[
(-1)^{k+1}D_kH_k=F_0+F_1H_k+\cdots+F_kH_k^k,
\]
for certain constants $F_i$. Thus, $H_{k}$ is locally constant on $\V$, which is a contradiction.

\textbf{(2.2)} If $\<\nabla H_{k},E_1\>\neq0$ or $\<\nabla H_{k},E_2\>\neq0$, then
by using (\ref{2E29A}) for $J=\{3,\dots,m\}$ (or $J=\emptyset$ if there is no number $m\geq 3$ such that (\ref{2I2}) holds) we obtain that
\[
(-1)^{k+1}D_kH_{k}=\mu_{_{k}}^{1,2,\ldots,m},
\]
that jointly with (\ref{p2}) implies again that $H_{k}$ is locally constant on $\V$, which is a contradiction. That concludes the proof in the Case 2.

\paragraph{Case 3: \emph{S} is of type III.}
Let $\{E_1,E_2,\ldots,E_n\}$ be a pseudo-orthonormal frame giving the canonical form of $S$ (see Proposition~\ref{caso3}). The gradient $\nabla H_{k}$ can be written, in this basis, as follows
\[
\nabla H_{k}=-\< \nabla H_{k},E_2\> E_1-\< \nabla H_{k},E_1\> E_2+\sum^n_{i=3}\< \nabla H_{k},E_i\> E_i,
\]
and then we obtain
\begin{align*}
P_{k}(\nabla H_{k})&=-(-1)^{k}\< \nabla H_{k},E_2\>\Big(\mu_{_{k}}^1E_1-\mu_{_{k-1}}^{1,2}E_2\Big)-(-1)^{k}\<\nabla H_{k},E_1\>\mu_{_{k}}^1E_2\\
&\quad+(-1)^{k}\sum^n_{i=3}\< \nabla H_{k},E_i\>\mu_{_{k}}^iE_i.
\end{align*}
Thus, equation (\ref{E26A}) yields the following system of equations on $\V$:
\begin{align*}
&\<\nabla H_{k},E_1\>\Big(D_kH_{k}+(-1)^{k}\mu_{_{k}}^1\Big)-\<\nabla H_{k},E_2\> (-1)^{k}\mu_{_{k-1}}^{1,2}=0,\\
&\<\nabla H_{k},E_2\>\Big(D_kH_{k}+(-1)^{k}\mu_{_{k}}^1\Big)=0,\\
&\<\nabla H_{k},E_i\>\Big(D_kH_{k}+(-1)^{k}\mu_{_{k}}^i\Big)=0,\qquad\textrm{for } i\geq3.
\end{align*}
Therefore, if $\<\nabla H_{k},E_2\>\neq 0$, then we get
\begin{equation}\label{3E27A}
(-1)^{k}\mu_{_{k}}^1=-D_kH_{k}\qquad\textrm{and}\qquad \mu_{_{k-1}}^{1,2}=0,	
\end{equation}
and, for every $i=3,\ldots,n$ such that $\<\nabla H_{k},E_i\>\neq0$, we have
\begin{equation}\label{3E27}
(-1)^{k}\mu_{_{k}}^i=-D_kH_{k}.
\end{equation}
We claim that $\<\nabla H_{k},E_i\>=0$ for some $i$. Otherwise, equations (\ref{3E27A}) and (\ref{3E27}) hold, and then we get
\begin{align*}
{\rm{tr}}(P_{k})&=-\< P_{k}E_1,E_2\>-\< P_{k}E_2,E_1\>+\sum^n_{i=3}\< P_{k}E_i,E_i\>\\
&=(-1)^{k}\mu_{_{k}}^1+(-1)^{k}\mu_{_{k}}^1+\sum_{i=3}^n(-1)^{k}\mu_{_{k}}^i=-nD_kH_{k}.	
\end{align*}
But this means, from Lemma~\ref{L1}(b), that $H_{k}=0$ on $\V$, which is a contradiction.

Observe that when $\<\nabla H_{k},E_i\>\neq0$, for some $i\geq 3$, then (after re-arranging the local orthonormal frame if necessary) we may assume that there exists some $m\in\{3,\ldots,n\}$ such that
\begin{align}\label{3I2}
\<\nabla H_{k},E_i\>\neq0& \textrm{ for } i=3,\ldots,m,\text{ and }\kappa_3<\cdots<\kappa_m.\\
\<\nabla H_{k},E_i\>=0& \textrm{ for } i=m+1,\ldots,n.\nonumber
\end{align}

\begin{claim}\label{claim4}
If $\<\nabla H_{k},E_1\>=\<\nabla H_{k},E_2\>=0$, then for every non-empty set $J\subseteq\{3,\ldots,m\}$ we have
\[
\mu_{_{k}}^J=(-1)^{k+1}D_kH_{k}.
\]
\end{claim}

\begin{claim}\label{claim5}
If $\<\nabla H_{k},E_1\>\neq0$ and $\< \nabla H_{k},E_2\>=0$, then for every non-empty set $J\subseteq\{1,3,\ldots,m\}$ we have
\[
\mu_{_{k}}^J=(-1)^{k+1}D_kH_{k}.
\]
\end{claim}

\begin{claim}\label{claim6}
If $\<\nabla H_{k},E_2\>\neq0$, then for every set $J\subseteq\{3,\ldots,m\}$ (admitting $J=\emptyset$) we have
\[
\mu_{_{k}}^{1,2,J}=(-1)^{k+1}D_kH_{k}.
\]
\end{claim}

In Claims \ref{claim5} and \ref{claim6}, $m\in\big\{3,\ldots,n\big\}$ is the number such that (\ref{3I2}) holds. If such number $m$ does not exist, then Claims \ref{claim5} and \ref{claim6} are only valid for $J=\{1\}$ and $J=\emptyset$, respectively. These Claims can be proved similarly to Claims \ref{claim1} and \ref{claim3}(b).

Finally, from (\ref{E18}) and (\ref{3I2}), we have that $AE_i=\eta_iE_i$, for $i=m+1,\ldots,n$,
where $\eta_i=-c_kH_{k+1}\kappa_i-cc_kH_k$ is a constant eigenvalue of the constant matrix $\vA$. Then we obtain that
\begin{equation}\label{p3}
\mu^{1,\ldots,m}_{_{r}}=B_0+B_1H_k+\cdots+B_rH_k^r,\qquad r>0,
\end{equation}
for certain constants $B_i$ (see (\ref{p2})). To finish the proof in this case, we distinguish three subcases:

\textbf{(3.1)} $\<\nabla H_{k},E_1\>=\<\nabla H_{k},E_2\>=0$,\\
\indent\textbf{(3.2)} $\<\nabla H_{k},E_1\>\neq0$ and $\<\nabla H_{k},E_2\>=0$,\\
\indent\textbf{(3.3)} $\<\nabla H_{k},E_2\>\neq0$.

Since the three subcases are similar, we will prove one of them, for example the case (3.1). From $\<\nabla H_{k},E_2\>=0$ and equation (\ref{E18}) we get $AE_2=\eta E_2$, where $\eta=-c_kH_{k+1}\kappa-cc_kH_k$ is a constant eigenvalue of $A$. Thus, we can find two constants $\beta_0,\beta_1$ such that
\begin{equation}\label{p4}
\kappa=\beta_0+\beta_1H_k.
\end{equation}
On the other hand,  from Claim~\ref{claim4} for the set $J=\{3,\ldots,m\}$, we obtain that
\[
(-1)^{k+1}D_kH_{k}=\mu^{3,\ldots,m}_{_{k}}=\mu^{1,\ldots,m}_{_{k}}+2\kappa\mu^{1,\ldots,m}_{_{k-1}}+\kappa^2\mu^{1,\ldots,m}_{_{k-2}},
\]
that, jointly with (\ref{p3}) and (\ref{p4}), leads to
\[
(-1)^{k+1}D_kH_{k}=G_0+G_1H_k+\cdots+G_kH_k^k,
\]
for certain constants $G_i$. Thus, $H_{k}$ is locally constant on $\V$, which is a contradiction.

Subcases (3.2) and (3.3) can be proved in a similar way by using now Claims \ref{claim5} and \ref{claim6}, respectively.

\paragraph{Case 4: \emph{S} is of type IV.}
Let $\{E_1,E_2,\ldots,E_n\}$ be a pseudo-orthonormal frame giving the canonical form of $S$ (see Proposition~\ref{caso4}). We proceed as in Case 3 to show that equation (\ref{E26A}) is equivalent to the following equations on $\V$:
\begin{align*}
&\<\nabla H_{k},E_1\>\Big(D_kH_{k}+(-1)^{k}\mu_{_{k}}^1\Big)+\<\nabla H_{k},E_3\>(-1)^{k}\mu_{_{k-1}}^{1,2}-\<\nabla H_{k},E_2\>(-1)^{k}\mu_{_{k-2}}^{1,2,3}=0,\\
&\<\nabla H_{k},E_2\>\Big(D_kH_{k}+(-1)^{k}\mu_{_{k}}^1\Big)=0,\\
&\<\nabla H_{k},E_3\>\Big(D_kH_{k}+(-1)^{k}\mu_{_{k}}^1\Big)-\<\nabla H_{k},E_2\>(-1)^{k}\mu_{_{k-1}}^{1,2}=0,\\
&\<\nabla H_{k},E_i\>\Big(D_kH_{k}+(-1)^{k}\mu_{_{k}}^i\Big)=0,\qquad\textrm{for } i\geq4.
\end{align*}
Thus, if $\<\nabla H_{k},E_2\>\neq 0$, we get
\begin{equation}\label{4E27A}
(-1)^{k}\mu_{_{k}}^1=-D_kH_{k},\qquad\mu_{_{k-1}}^{1,2}=0\qquad\textrm{and}\qquad \mu_{_{k-2}}^{1,2,3}=0.	
\end{equation}
However, if $\<\nabla H_{k},E_2\>=0$ and $\<\nabla H_{k},E_3\>\neq0$, then we have
\begin{equation}\label{4E27B}
(-1)^{k}\mu_{_{k+1}}^1=-D_kH_{k}\qquad\textrm{and}\qquad\mu_{_{k-1}}^{1,2}=0.	
\end{equation}
Moreover, for every $i=4,\ldots,n$ such that $\<\nabla H_{k},E_i\>\neq0$, we get
\begin{equation}\label{4E27}
(-1)^{k}\mu_{_{k}}^i=-D_kH_{k}.
\end{equation}
We claim that $\<\nabla H_{k},E_i\>=0$ for some $i$. Otherwise, (\ref{4E27A}) and (\ref{4E27}) hold and then we deduce
\begin{align*}
{\rm{tr}}(P_{k})&=-\< P_{k}E_1,E_2\>-\< P_{k}E_2,E_1\>+\< P_{k}E_3,E_3\>+\sum^n_{i=4}\< P_{k}E_i,E_i\>\\
&=(-1)^{k}\mu_{_{k}}^1+(-1)^{k}\mu_{_{k}}^1+(-1)^{k}\mu_{_{k}}^1+\sum_{i=4}^n(-1)^{k}\mu_{_{k}}^i=-nD_kH_{k},
\end{align*}
but this means, from Lemma \ref{L1}, that $H_{k}=0$ on $\V$, which is a contradiction.

If there is some $i\geq 4$ such that $\<\nabla H_{k},E_i\>\neq0$, then (after re-arranging the local pseudo-orthonormal frame if necessary) we may assume that there exists some number $m\in\{4,\dots,n\}$ such that
\begin{align}\label{4I2}
\<\nabla H_{k},E_i\>\neq0& \textrm{ for } i=4,\ldots,m,\text{ and }\kappa_4<\cdots<\kappa_m.\\
\<\nabla H_{k},E_i\>=0& \textrm{ for } i=m+1,\ldots,n.\nonumber
\end{align}

\begin{claim}\label{claim7}
If $\<\nabla H_{k},E_1\>=\<\nabla H_{k},E_2\>=\<\nabla H_{k},E_3\>=0$, then for every non-empty set $J\subseteq\{4,\ldots,m\}$ we have
\[
\mu_{_{k}}^J=(-1)^{k+1}D_kH_{k}.
\]
\end{claim}

\begin{claim}\label{claim8}
If $\<\nabla H_{k},E_1\>\neq0$ and $\< \nabla H_{k},E_2\>=\< \nabla H_{k},E_3\>=0$, then for every
set $J\subseteq\{4,\ldots,m\}$ we have
\[
\mu_{_{k}}^{1,J}=(-1)^{k+1}D_kH_{k}.
\]
\end{claim}

\begin{claim}\label{claim9}
If $\<\nabla H_{k},E_2\>=0$ and $\<\nabla H_{k},E_3\>\neq0$, then for every set $J\subseteq\{4,\ldots,m\}$ we have
\[
\mu_{_{k}}^{1,2,J}=(-1)^{k+1}D_kH_{k}.
\]
\end{claim}

\begin{claim}\label{claim10}
If $\<\nabla H_{k},E_2\>\neq0$, then for every set $J\subseteq\{4,\ldots,m\}$ we have
\begin{equation}\label{4E29D}
\mu_{_{k}}^{1,2,3,J}=(-1)^{k+1}D_kH_{k}.
\end{equation}
\end{claim}

In Claims \ref{claim8}, \ref{claim9} and \ref{claim10}, $m\geq4$ is the number such that (\ref{4I2}) holds. If such number does not exist, then these claims are valid only for $J=\emptyset$.

Claims \ref{claim7}, \ref{claim8} and \ref{claim9} can be proved analogously to Claims \ref{claim1} and \ref{claim3}. Thus we are going to prove Claim \ref{claim10} by induction on the cardinality of $J$.

From (\ref{4E27A}) we get $\kappa\mu_{_{k-2}}^{1,2,3}+\mu_{_{k-1}}^{1,2,3}=0$, and then
\[
\mu_{_{k}}^1=\mu_{_{k}}^{1,2}+\kappa\mu_{_{k-1}}^{1,2}=\mu_{_{k}}^{1,2}=\mu_{_{k}}^{1,2,3}+\kappa\mu_{_{k-1}}^{1,2,3}
=\mu_{_{k}}^{1,2,3}+\kappa\Big(\!\!\!-\kappa\mu_{_{k-2}}^{1,2,3}\Big)=\mu_{_{k}}^{1,2,3}.	
\]
From here and the first equation of (\ref{4E27A}), we obtain (\ref{4E29D}) for $\card(J)=0$ (i.e. $J=\emptyset$). Let us assume now that (\ref{4E29D})  hold for every set $J$ with $\card(J)=0,1,\ldots,p<m-3$ and take a set $J_0=\big\{j_{1},j_2\ldots,j_{p+1}\big\}\subseteq\big\{4,\ldots,m\big\}$ with cardinality $p+1\leq m-3$. Let $J_1$ and $J_2$ be the two sets of cardinality $p$ such that
\[
J_0=\underbrace{\big\{j_{1},j_3,\ldots,j_{p+1}\big\}}_{J_2}\cup\big\{j_{2}\big\}=\underbrace{\big\{j_{2},j_3,\ldots,j_{p+1}\big\}}_{J_1}\cup\big\{j_{1}\big\}.
\]
By applying the induction hypothesis to $J_1$ and $J_2$, we deduce
\[
\mu_{_{k}}^{1,2,3,J_1}=\mu_{_{k}}^{1,2,3,J_2}=(-1)^{k+1}D_kH_{k}.
\]
By using (\ref{F}) in the first equality we get $(\kappa_{j_{1}}-\kappa_{j_2})\mu_{_{k-1}}^{1,2,3,J_0}=0$, and from (\ref{4I2}) we obtain $\mu_{_{k-1}}^{1,2,3,J_0}=0$. Thus $\mu_{_{k}}^{1,2,3,J_1}=\mu_{_{k}}^{1,2,3,J_2}=\mu_{_{k}}^{1,2,3,J_0}$, and the proof of the Claim finishes.

Finally, from (\ref{E18}) and (\ref{4I2}), we get $AE_i=\eta_iE_i$, with $i=m+1,\ldots,n$, where
$\eta_i=-c_kH_{k+1}\kappa_i-cc_kH_k$ is a constant eigenvalue of the constant matrix $A$. Thus we have
\begin{equation}\label{p6}
\mu^{1,\ldots,m}_{_{r}}=B_0+B_1H_k+\cdots+B_rH_k^r,
\end{equation}
for certain constants $B_i$ (see (\ref{p2})). To finish the proof in this Case, we distinguish four subcases:

\textbf{(4.1)} $\<\nabla H_{k},E_1\>=\<\nabla H_{k},E_2\>=\<\nabla H_{k},E_3\>=0$,\\
\indent\textbf{(4.2)} $\<\nabla H_{k},E_1\>\neq0$ and $\<\nabla H_{k},E_2\>=\<\nabla H_{k},E_3\>=0$,\\
\indent\textbf{(4.3)} $\<\nabla H_{k},E_2\>=0$ and $\<\nabla H_{k},E_3\>\neq0$,\\
\indent\textbf{(4.4)} $\<\nabla H_{k},E_2\>\neq0$.

Since the four subcases are similar, we are going to prove one of them, the case (4.1). From $\<\nabla H_{k},E_2\>=0$ and (\ref{E18}) we get $AE_2=\eta E_2$, where $\eta=-c_kH_{k+1}\kappa-cc_kH_k$ is a constant eigenvalue of $A$. Then we have
\begin{equation}\label{p7}
\kappa=\beta_0+\beta_1H_k,
\end{equation}
for certain constants $\beta_0$ and $\beta_1$. On the other hand, from Claim \ref{claim7} for the set $J=\{4,\ldots,m\}$, we obtain
\[
(-1)^{k+1}D_kH_{k}=\mu^{4,\ldots,m}_{_{k}}
=\mu^{1,\ldots,m}_{_{k}}+3\kappa\mu^{1,\ldots,m}_{_{k-1}}+3\kappa^2\mu^{1,\ldots,m}_{_{k-2}}+\kappa^3\mu^{1,\ldots,m}_{_{k-3}},
\]
that, jointly with (\ref{p6}) and (\ref{p7}), leads to
\[
(-1)^{k+1}D_kH_{k}=G_0+G_1H_k+\cdots+G_kH_k^k,
\]
for certain constants $G_i$. But this means that $H_{k}$ is locally constant on $\V$, which is a contradiction.

Subcases (4.2), (4.3) and (4.4) can be proved similarly by using now Claims \ref{claim8}, \ref{claim9} and \ref{claim10}, respectively.

In conclusion, from Cases 1--4 we deduce that the $k$-th mean curvature $H_k$ is constant, and so the proof of Lemma \ref{L7} finishes.
\end{proof}

\section{Proof of Theorem~\ref{T1}}
\label{s:proof.T1}

We have already checked in Section \ref{s:examples} that each one of the hypersurfaces mentioned in Theorem \ref{T1} does satisfy the condition $L_k\psi=A\psi$, for a self-adjoint constant matrix $A\in\mathbb{R}^{(n+2)\times(n+2)}$. Conversely, let us assume that $\psi:M\rightarrow\mathbb{M}_c^{n+1}\subset\mathbb{R}_q^{n+2}$ satisfies the condition $L_k\psi=A\psi$, for some self-adjoint constant matrix $A\in\mathbb{R}^{(n+2)\times(n+2)}$. Since $b=0$, from (\ref{24}) we get that $H_k$ is constant on $M$, and from Lemma \ref{L7} we know that $H_{k+1}$ is also constant on $M$.

Let us assume that $H_{k+1}$ is a non-zero constant (otherwise, there is nothing to prove).
From (\ref{E13}), (\ref{E19A}) and (\ref{E18}), we have
\begin{align}
A\psi&=c_kH_{k+1}N-cc_kH_k\psi,\label{E31+}\\
\vA X&=-c_kH_{k+1}SX-cc_kH_kX,\label{E32}\\
AN&=\alpha N +\e cc_kH_{k+1}\psi,\label{E31}
\end{align}
with $\alpha=-\varepsilon C_k(nH_1H_{k+1}-(n-k-1)H_{k+2})$. Taking covariant derivative in (\ref{E31}), and using (\ref{E32}), we have
\[
\nabla^0_X(AN)=\<\nabla\alpha,X\>N-\alpha SX + \e cc_kH_{k+1}X,
\]
but also from (\ref{E32}) we obtain
\[
\nabla^0_X(AN)=A(\nabla^0_XN)=-A(SX)=c_kH_{k+1}S^2X+cc_kH_kSX.
\]
From the last two equations we deduce that $\alpha$ is constant on $M$, and also that the shape operator $S$ satisfies the equation
\begin{equation}\label{S.th1}
S^2+\lambda S-\e cI=0,\qquad \lambda=\frac{\alpha}{c_kH_{k+1}}+\frac{cH_k}{H_{k+1}}=\textrm{constant.}
\end{equation}
As a consequence, $M$ is an isoparametric hypersurface in $\mathbb{M}_c^{n+1}$ and the minimal polynomial of its shape operator $S$ is of degree at most two. We claim that $M$ is not totally umbilical. Otherwise, from Example \ref{ej2} we get that it should be totally geodesic, but this is a contradiction, since we are supposing that $H_{k+1}$ is a non-zero constant. Thus, the minimal polynomial of $S$ is exactly of degree two. If $S$ is diagonalizable, then $M$ has exactly two distinct constant principal curvatures, and then it is an open piece of a standard pseudo-Riemannian product (Example \ref{ej3}), \cite{Nom81,X99,ZX}.

Suppose now that $S$ is not diagonalizable, so that the minimal polynomial of $S$ is given by $\mu_S(t)=t^2+\lambda t-\e c$, with discriminant $d_S=\lambda^2+4\e c\leq 0$. From equations (\ref{E31+})--(\ref{S.th1}) we easily deduce that the minimal polynomial of $A$ is given by $\mu_A(t)=t^2+a_1t+a_0$,
where $a_1=2cc_kH_k-\lambda c_kH_{k+1}$ and $a_0=c_k^2H_k^2-\lambda cc_k^2H_kH_{k+1}-\e c c_k^2H_{k+1}^2$ are constants. Since the discriminant $d_A$ of $\mu_A(t)$ is given by $d_A=c_k^2H_{k+1}^2d_S$,  then $A$ also is not diagonalizable. Since $\<A\psi,\psi\>=-c_k$ is constant and $\mu_A(-cc_k)\neq0$, then $M$ is an open piece of a quadratic hypersurface as in Example \ref{ej4}. That concludes the proof of Theorem \ref{T1}.

\section{Proof of Theorem~\ref{T2}}
\label{s:proof.T2}

We have already checked in Section \ref{s:examples} that each one of the hypersurfaces mentioned in Theorem \ref{T2} does satisfy the condition $L_k\psi=A\psi+b$, for a self-adjoint constant matrix $A\in\mathbb{R}^{(n+2)\times(n+2)}$ and some non-zero constant vector $b$. Conversely, let us assume that $\psi:M\rightarrow\mathbb{M}_c^{n+1}\subset\mathbb{R}_q^{n+2}$ satisfies the condition $L_k\psi=A\psi+b$, for some self-adjoint constant matrix $A\in\mathbb{R}^{(n+2)\times(n+2)}$ and some non-zero constant vector $b$. Since $H_k$ is assumed to be constant on $M$, from Lemma \ref{L7} we know that $H_{k+1}$ is also constant on $M$. The case $H_{k+1}=0$ cannot occur, because in that case we have $b=0$ (see Example \ref{ej1}).

Let us assume that $H_{k+1}$ is a non-zero constant. From (\ref{24}) we obtain that $b^\top=0$ and that the function $\<b,\psi\>$ is constant on $M$. Now we use (\ref{23}) to deduce that
\[
\<b,N\>=\frac{cH_k}{H_{k+1}}\<b,\psi\>=\text{constant}.
\]
Since $b=\e\<b,N\>N+c\<b,\psi\>\psi$, taking covariant derivative in this equation we have
\[
-\e\<b,N\>SX+c\<b,\psi\>X=0,
\]
for any tangent vector field $X$. If $\<b,N\>\neq 0$, then $M$ is totally umbilical (but not totally geodesic). Otherwise, $b=c\<b,\psi\>\psi$ and then $\<b,\psi\>=0$, but this implies $b=0$. That concludes the proof of Theorem \ref{T2}.

\end{document}